\documentclass[11pt]{amsart}

\makeatletter

\renewcommand\subsubsection{\@secnumfont}{\bfseries\itshape}%
\renewcommand\subsubsection{\@startsection{subsubsection}{3}
	\z@{.5\linespacing\@plus.7\linespacing}{-.5em}%
	{\normalfont\bfseries\itshape}}

\makeatother

\usepackage{etoolbox}

\makeatletter
\patchcmd{\@setaddresses}{\indent}{\noindent}{}{}
\patchcmd{\@setaddresses}{\indent}{\noindent}{}{}
\patchcmd{\@setaddresses}{\indent}{\noindent}{}{}
\patchcmd{\@setaddresses}{\indent}{\noindent}{}{}
\makeatother

\usepackage{graphicx}
\usepackage{amssymb}
\usepackage{amsmath}
\usepackage{amsthm}
\usepackage{colortbl}
\usepackage{blkarray, bigstrut} %
\usepackage{xcolor}
\usepackage{bm}
\usepackage{placeins}
\usepackage{verbatim}
\usepackage{multirow}
\usepackage{float}
\usepackage{enumitem}
\usepackage[a4paper,left=1.3in,right=1.3in,top=1.3in,bottom=1.3in]{geometry}
\usepackage{url}
\usepackage{enumitem}
\usepackage{tikz-cd}
\usepackage[all]{xy}
\usepackage{url}
\usepackage{array,booktabs}
\usepackage{algorithmicx}
\usepackage{algpseudocode}
\usepackage{algorithm}
\usepackage[colorlinks=false, hidelinks]{hyperref}
\usepackage{etex}
\usepackage{tikz}
\usepackage{nicematrix}
\usepackage[width=0.9\textwidth]{caption}
\usetikzlibrary{decorations.pathmorphing} 
\usetikzlibrary{arrows}
\usetikzlibrary{shapes}

\newcommand*\circled[1]{\tikz[baseline=(char.base)]{
		\node[shape=circle,draw,inner sep=2pt] (char) {#1};}}
\makeatletter
\DeclareRobustCommand*{\bfseries}{%
	\not@math@alphabet\bfseries\mathbf
	\fontseries\bfdefault\selectfont
	\boldmath
}
\makeatother
\usepackage{xargs}                      
\usepackage[colorinlistoftodos,prependcaption,textsize=tiny]{todonotes}
\newcommandx{\change}[2][1=]{\todo[linecolor=red,backgroundcolor=red!25,bordercolor=red,#1]{#2}}
\newcommandx{\unsure}[2][1=]{\todo[linecolor=blue,backgroundcolor=blue!25,bordercolor=blue,#1]{#2}}
\newcommandx{\info}[2][1=]{\todo[linecolor=OliveGreen,backgroundcolor=OliveGreen!25,bordercolor=OliveGreen,#1]{#2}}
\newcommandx{\improvement}[2][1=]{\todo[linecolor=Plum,backgroundcolor=Plum!25,bordercolor=Plum,#1]{#2}}
\newcommandx{\thiswillnotshow}[2][1=]{\todo[disable,#1]{#2}}
\makeatletter
\providecommand\@dotsep{5}
\renewcommand{\listoftodos}[1][\@todonotes@todolistname]{%
	\@starttoc{tdo}{#1}}
\makeatother

\numberwithin{equation}{section} 

\newtheorem{theorem}[equation]{Theorem}

\newtheorem*{thm:main}{Main Theorem}
\newtheorem{fact}[equation]{Fact}
\newtheorem{proposition}[equation]{Proposition}
\newtheorem{lemma}[equation]{Lemma}
\newtheorem{corollary}[equation]{Corollary}

\newtheorem*{warning}{Warning}

\theoremstyle{definition}
\newtheorem{definition}[equation]{Definition}

\newtheorem{remark}[equation]{Remark}

\newtheorem*{remark*}{Remark}
\newtheorem*{question}{Question}

\newcommand{\ZH}{\zz \lbrack H \rbrack}

\newcommand{\bez}{-}

\newcommand{\zz}{\mathbb{Z}}

\newcommand{\bigslant}[2]{{\raisebox{.2em}{$#1$}\left/\raisebox{-.2em}{$#2$}\right.}}
\newcommand{\Mab}{M^{fab}}

\newcommand{\Vab}{\mathcal{V}^{fab}}

\newcommand{\V}{\mathcal{V}}
\newcommand{\B}{\mathcal{B}}
\newcommand{\T}{\mathcal{T}}

\newcommand{\D}{\mathcal{D}}

\algrenewcommand{\algorithmiccomment}[1]{\hspace*{\fill}
	\color{gray}\small $\#$  #1 \color{black}\normalsize}

\makeatletter
\@namedef{subjclassname@2020}{%
	\textup{2020} Mathematics Subject Classification}
\makeatother
\usepackage{doc}
\makeatletter
\def\maketitle{\par
	\begingroup
	\def\thefootnote{\fnsymbol{footnote}}%
	\setcounter{footnote}\z@
	\def\@makefnmark{\hbox to\z@{$\m@th^{\@thefnmark}$\hss}}%
	\long\def\@makefntext##1{\noindent
		\ifnum\c@footnote>\z@\relax
		\hbox to1.8em{\hss$\m@th^{\@thefnmark}$}##1%
		\else
		\hbox to1.8em{\hfill}%
		\parbox{\dimexpr\linewidth-1.8em}{\raggedright ##1}%
		\fi}
	\if@twocolumn\twocolumn[\@maketitle]%
	\else\newpage\global\@topnum\z@\@maketitle\fi
	\thispagestyle{titlepage}\@thanks\endgroup
	\setcounter{footnote}\z@
	\gdef\@date{\today}\gdef\@thanks{}%
	\gdef\@author{}\gdef\@title{}}

\makeatother
\begin{document}
	\hypersetup{%
		,urlcolor=black
		,citecolor=black
		,linkcolor=black
	}

	\numberwithin{equation}{section}
	\title{Arbitrarily large veering triangulations with \\ a vanishing taut polynomial}
	
	\author{Anna Parlak}
	\address{Mathematical Sciences Building \\ University of California, One Shields Avenue, Davis, CA 95616, United States}
	\email{anna.parlak@gmail.com}


	\begin{abstract}
	Landry, Minsky, and Taylor introduced an invariant of veering triangulations called the taut polynomial. Via a connection between veering triangulations and pseudo-Anosov flows, it generalizes the Teichm\"uller polynomial of a fibered face of the Thurston norm ball to (some) non-fibered faces. We construct a sequence of veering triangulations, with the number of tetrahedra tending to infinity, whose taut polynomials vanish. These veering triangulations encode non-circular Anosov flows transverse to tori.
	\end{abstract}
	
	\maketitle%
	\setcounter{tocdepth}{1}
	
	\tableofcontents

	\section{Introduction}
	Veering triangulations were introduced by Ian Agol in \cite{Agol_veer} as a way to canonically triangulate certain pseudo-Anosov mapping tori. It quickly became evident that veering triangulations exist also on non-fibered 3-manifolds \cite[Section 4]{veer_strict-angles}. In unpublished work Agol and Gu\'eritaud showed that a veering triangulation can be constructed from a pair $(\Psi, \Lambda)$, where $\Psi$ is a pseudo-Anosov flow on a closed 3-manifold, and $\Lambda$ is a finite, nonempty collection of closed orbits of $\Psi$ containing all singular orbits of $\Psi$ and such that $\Psi$ does not have \emph{perfect fits} relative to $\Lambda$; see \cite[Section 4]{LMT_flow}. The original fibered setup discussed in \cite{Agol_veer} is just a special case in which $\Psi$ is the suspension flow of a pseudo-Anosov homeomorphism, and~$\Lambda$ consists only of the sigular orbits of~$\Psi$. 
	
	The connection between veering triangulations and pseudo-Anosov flows has been explored further by Agol-Tsang \cite{Tsang-Agol}, Landry-Minsky-Taylor \cite{LMT_flow}, and Schleimer-Segerman \cite{SchleimSeg-dynamic-pairs, SchleimSegLinks}. From their work it follows that veering triangulations can be viewed as combinatorial tools to study pseudo-Anosov flows. They have numerous advantages: first, as finite combinatorial objects they are often easier to analyze than the flows directly. In particular, veering triangulations simplify various computations of flow invariants; see for instance \cite{Parlak-computation}.  Second, veering triangulations can be rigorously classified and catalogued. This feature prompted  the creation of the  Veering Census~\cite{VeeringCensus} which contains information on all veering triangulations with up to 16 tetrahedra. Third, a veering triangulation constructed from a pair $(\Psi, \Lambda)$ is canonical for that pair. This makes veering triangulations superior to other combinatorial methods used to study pseudo-Anosov flows, such as Markov partitions and branched surfaces. Finally, thanks to veering triangulations one can study flows experimentally using already existing  software to study triangulations, such as Regina~\cite{regina} and SnapPy~\cite{snappea}.
	
	New results about pseudo-Anosov flows that were proved using veering triangulations include the construction of Birkhoff sections for pseudo-Anosov flows with the complexity bounded from above \cite[Theorem 6.3]{Tsang-Birkhoff}, and the fact that a top-dimensional non-fibered face of the Thurston norm ball can be dynamically represented by multiple topologically inequivalent flows \cite[Theorem 5.2]{Parlak-mutations}.
	
	\subsection{The main result} In this paper we do not use veering triangulations to prove new facts about pseudo-Anosov flows. Instead, we focus on veering triangulations themselves. In \cite[Section~3]{LMT} Landry-Minsky-Taylor introduced a polynomial invariant of veering triangulations called the \emph{taut polynomial}. It carries information about the dynamical properties of the flow encoded by the veering triangulation; see \cite[Section~7]{LMT} for the fibered case and \cite[Section 7]{LMT_flow} for the non-fibered case. 
		Interestingly, the taut polynomial of a veering triangulation which encodes a non-circular pseudo-Anosov flow can vanish. In fact, we prove the following theorem.

	\begin{thm:main}
		For any integer $k\geq0$ there is a veering triangulation $\V_k$ with $8+k$ tetrahedra whose taut polynomial vanishes.
	\end{thm:main}

By a direct computation of the taut polynomials of all veering triangulations with less than 8 tetrahedra it can be checked that 8 is the minimum number of tetrahedra that a veering triangulation with a vanishing taut polynomial can have.

To construct veering triangulations $(\V_k)_{k\geq 0}$ from the Main Theorem we use \emph{vertical surgeries} introduced by Tsang in \cite[Section 4]{Tsang-branched-surfaces}. Topologically, vertical surgeries are just $(1,k)$-Dehn surgeries along certain loops embedded in the 3-manifold underlying a veering triangulation. The loops are chosen so that the surgered 3-manifold admits a veering triangulation closely related to the original one; see Section \ref{sec:vertical:surgery}. To prove the Main Theorem we  perform vertical surgeries on some veering triangulation~$\V_0$  whose taut polynomial  vanishes, and show that all surgered triangulations also have a vanishing taut polynomial; see Proposition \ref{prop:taut:vanishes}. We discuss the flows encoded by $(\V_k)_{k\geq 0}$ in Section \ref{sec:the:flows}.

 The sequence $(\V_k)_{k \geq 0}$ from the Main Theorem is not unique. In fact, starting from the same veering triangulation $\V_0$ one can construct a different sequence of veering triangulations with vanishing taut polynomials by choosing a different vertical surgery curve; see Remark \ref{rem:other:sequence}. 


The stable branched surfaces of veering triangulations $(\V_k)_{k\geq 0}$ from the Main Theorem are not orientable. However, lifting every $\V_k$ to the double cover where its stable branched surface becomes orientable yields an infinite sequence $(\widehat{\V}_k)_{k\geq 0}$  of veering triangulations with the number of tetrahedra tending to infinity, orientable stable branched surfaces, and a vanishing taut polynomial; see Remark~\ref{rem:double:covers}. By \cite[Theorem 5.7]{parlak-alex}, the Alexander polynomials of manifolds underlying $(\widehat{\V}_k)_{k\geq 0}$ also vanish. The Alexander polynomials of manifolds underlying $(\V_k)_{k\geq 0}$ do not vanish.

	\subsection{Vanishing veering polynomials}
 In \cite[Section~3]{LMT} Landry-Minsky-Taylor introduced another polynomial invariant of veering triangulations called the \emph{veering polynomial}. Since the taut polynomial of a veering triangulation~$\V$ is a factor of the veering polynomial of~$\V$ \cite[Theorem 6.1 and Remark~6.18]{LMT}, the veering polynomials of veering triangulations~$(\V_k)_{k \geq 0}$ from the Main Theorem also vanish. That is, we get the following corollary.

\begin{corollary}
	For any integer $k\geq0$ there is a veering triangulation $\V_k$ with $8+k$ tetrahedra whose veering polynomial vanishes.\qed
\end{corollary}

	\subsection{Connection to the finiteness conjecture}
	
	Let $\V$ be a finite veering triangulation of a 3-manifold $M$. In \cite[Theorem 5.7]{parlak-alex} the author showed that the taut polynomial of $\V$  is equal to the twisted Alexander polynomial $\Delta_M^{\omega \otimes \pi},$ where \linebreak $\omega: \pi_1(M)\rightarrow \lbrace -1, 1 \rbrace$ is a certain homomorphism determined by $\V$ (see Section \ref{sec:computation:fox}) and \mbox{$\pi: \pi_1(M) \rightarrow H_1(M;\zz)/\text{torsion}$} is the abelianization and torsion-killing projection. Since there are only $|H_1(M;\zz/2)|$ homomorphisms $\pi_1(M) \rightarrow \zz/2$, it follows that for any 3-manifold $M$ there are only finitely many candidates for the taut polynomial of a veering triangulation of $M$. One could hope that by computing these potential taut polynomials it is possible to obtain some information about the potential veering triangulations of $M$. This would be interesting because it is still not known how to tell whether a given cusped hyperbolic 3-manifold admits a veering triangulation, and if so --- how many. In particular, it is conjectured that any 3-manifold admits only finitely such  triangulations. Therefore bounding from above the number of tetrahedra  of a veering triangulation of $M$ would be of particular interest. Unfortunately, the Main Theorem suggests that the knowledge of all twisted Alexander polynomials $\Delta_M^{\omega \otimes \pi}$, for all homomorphisms  $\omega: \pi_1(M)\rightarrow \zz/2$, might not reveal much information about the maximal size of  a veering triangulation of $M$.
	
	
\subsection*{Acknowledgements}
The author thanks Ian Agol for asking whether infinitely many different veering triangulations can have the same taut polynomial.

	\section{Veering triangulations, the taut polynomial, and vertical surgeries}
	Let $M$ be a compact, oriented 3-manifold. By an \emph{ideal triangulation} of $M$ we mean a decomposition of $M \bez \partial M$ into a collection of finitely many ideal tetrahedra with triangular faces identified in pairs by homeomorphisms which send vertices to vertices. If $\T$ is an ideal triangulation of $M$ then its dual 2-complex is called its \emph{dual spine}. We denote it by $\D$. A vertex of $\D$ that is dual to a tetrahedron $t$ of $\T$ is denoted by $\nu(t)$. An edge of $\D$ that is dual to a face $f$ of $\T$ is denoted by $\epsilon(f)$. A 2-cell of $\D$ that is dual to an edge $e$ of $\T$ is denoted by $\sigma(e)$. 
	The 1-skeleton of the dual spine is called the \emph{dual graph} of $\T$ and denoted by $\Gamma$.
	
	\subsection{Taut triangulations}
	In \cite[Introduction]{Lack_taut} Lackenby introduced a subclass of ideal triangulations that are called \emph{taut}. We define tautness of an ideal triangulation in terms of properties of its dual spine as follows.

\begin{definition}\label{def:taut}
		A \emph{taut structure} $\alpha$ on an ideal triangulation $\T$ is a choice of orientations on the edges of its dual graph $\Gamma$ such that 
		\begin{enumerate}
			\item every vertex $\nu$ of $\Gamma$ has two incoming edges and two outgoing edges,
			\item every 2-cell $\sigma$ of $\D$ has one vertex $b_\sigma$ such that the two edges of $\sigma$ adjacent to~$b_\sigma$ both point out of $b_\sigma$,\nopagebreak[10000]
			\item every 2-cell $\sigma$ of $\D$ has one vertex $t_\sigma$ such that the two edges of~$\sigma$ adjacent to~$t_\sigma$ both point into $t_\sigma$.
		\end{enumerate}
	\end{definition}
	A \emph{taut triangulation} is a pair $(\T, \alpha)$, where $\T$ is an ideal triangulation, and $\alpha$ is a taut structure on $\T$. The vertex $b_\sigma$ from Definition \ref{def:taut} (2) is called the \emph{bottom vertex} of $\sigma$, and the vertex $t_\sigma$ from Definition \ref{def:taut} (2) is called the \emph{top vertex} of $\sigma$. 
	
	\begin{remark}
		Taut triangulations are often called \emph{transverse taut triangulations}; see for instance \cite{Parlak-computation, parlak-alex, SchleimSegLinks}.
	\end{remark}

Under the duality between $\T$ and $\D$, a taut structure $\alpha$ can be seen as a choice of coorientations on 2-dimensional faces of $\T$. If $(\T, \alpha)$ is taut then, by Definition \ref{def:taut}(1), every  tetrahedron $t$ of $\T$ has two faces whose coorientations point into $t$, and two faces whose coorientations point out of $t$. We call the pair of faces whose coorientations point out of $t$ the \emph{top faces} of~$t$ and the pair of faces whose coorientations point into~$t$ the \emph{bottom faces} of $t$. The common edge of the top faces is called the \emph{top diagonal} of $t$, and the common edge of the bottom faces is called the \emph{bottom diagonal} of $t$.
We encode a taut structure on a tetrahedron by drawing it as a quadrilateral with two diagonals --- one on top of the other; see Figure~\ref{fig:taut_tet}. Then the convention is that coorientations on all faces point towards the reader. In other words, we view the tetrahedron \emph{from above}.

\begin{figure}[h]
	\includegraphics[scale=1]{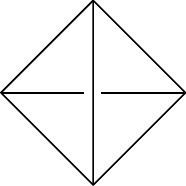}
	\caption{Taut tetrahedron.} 
	\label{fig:taut_tet}
\end{figure}

	\subsection{Veering triangulations} \label{sec:veering:tri}
	Among taut ideal triangulation we distinguish a further subclass of triangulations that are called \emph{veering}. 
\begin{definition}\label{def:veering}
	A \emph{veering structure} on a taut ideal triangulation is a smoothening of its dual spine into a branched surface $\B$ which locally around every vertex looks either as in Figure \ref{fig:veering_branched_surface} (a) or Figure \ref{fig:veering_branched_surface} (b). 

\begin{figure}[h]
	\includegraphics[scale=0.6]{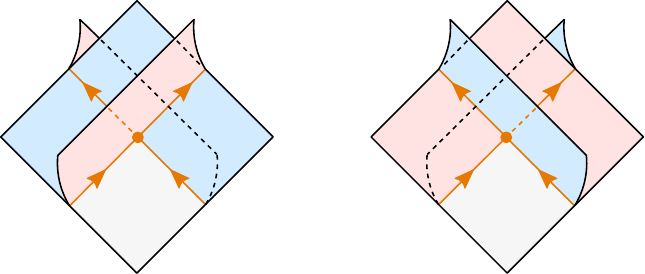}
	\put(-195,0){(a)}
	\put(-90, 0){(b)}
	\caption{A local picture of a neighborhood of a vertex of a veering branched surface. Orientation on the edges of the branch locus is determined by the taut structure. }
	\label{fig:veering_branched_surface}
\end{figure}
A \emph{veering triangulation} is a taut ideal triangulation with a veering structure.
\end{definition} 

We call the branched surface $\B$ from Definition \ref{def:veering} the \emph{stable branched surface} of a veering triangulation or a \emph{veering branched surface}. 
Its branch locus can be identified with the dual graph $\Gamma$ of the triangulation, and it is oriented by the taut structure. 
The 2-cells of $\B$ are called its \emph{sectors}. It follows from the definition of a taut triangulation (Definition \ref{def:taut}) that if $\sigma$ is a sector of $\B$ then $\sigma$ is a disk whose top and bottom vertices partition the boundary $\partial \sigma$ into two oriented arcs. We call these arcs the two \emph{side-arcs} of $\sigma$. 
We say that an edge $\epsilon$ in the boundary of $\sigma$ is its \emph{uppermost edge} if its terminal endpoint is the top vertex of $\sigma$. We refer to all other edges of $\sigma$ as its \emph{lower edges}. Lower edges of $\sigma$ whose initial vertex is the bottom vertex of $\sigma$ are called its \emph{lowermost edges}.

An edge $\epsilon$ of $\B$ is embedded  in the boundary of three (not necessarily distinct) sectors. One sector $\sigma$ in which $\epsilon$ is embedded has the property that  for every other sector~$\sigma'$ in which $\epsilon$ is embedded there is a smooth simple arc which connects an interior point of~$\sigma$ with an interior point of $\sigma'$, and intersects the 1-skeleton of $\B$ in one point contained in the interior of $\epsilon$. We say that~$\sigma$ is on the \emph{one-sheeted side} of $\epsilon$. The remaining two sectors in which~$\epsilon$ is embedded are on the \emph{two-sheeted side} of $\epsilon$. 
For an observer looking from~$\epsilon$ towards the two-sheeted side of $\epsilon$ (with feet at the initial endpoint of $\epsilon$ and eyes at the terminal endpoint of~$\epsilon$) one of the sectors on this side veers off to the right, and the other veers off to the left. We call the former sector the \emph{right sector} of $\epsilon$, and the latter --- the \emph{left sector} of~$\epsilon$. 

By $\epsilon^{-1}$ we denote the edge $\epsilon$ with the orientation opposite to that determined by the taut structure. Given edges $\epsilon_1, \epsilon_2$ of a veering branched surface we say that a pair $\left( \epsilon_1^{\delta_1}, \epsilon_2^{\delta_2}\right)$, where $\delta_j = \pm1$ for $j=1,2$, forms a \emph{turn} if the terminal endpoint of $\epsilon_1^{\delta_1}$ is the initial endpoint of $\epsilon_2^{\delta_2}$. All types of turns are illustrated in Figure \ref{fig:turns}. A turn $\left( \epsilon_1^{\delta_1}, \epsilon_2^{\delta_2}\right)$ is \emph{positively oriented} if $\delta_1 = \delta_2 = 1$; in this case we skip the superscripts and write just $(\epsilon_1, \epsilon_2)$. 
  Using this terminology, below we state important features of a veering branched surface that will be used in Lemmas \ref{lem:colors:on:the:1:sheeted:side} and \ref{lem:vertical:surgery:curves}.

\begin{lemma}\label{lem:direction:well:defined}
	Let $\sigma$ be a sector of a veering branched surface.
	\begin{enumerate}
		\item Suppose that $(\epsilon, \epsilon')$ is a positively oriented turn in the boundary of $\sigma$. Then either $\sigma$ is on the one-sheeted side of $\epsilon'$ and on the two-sheeted side of $\epsilon$, or $\sigma$ is on the two-sheeted side of both $\epsilon, \epsilon'$, and it is the right (left) sector of $\epsilon$ if and only if it is the right (left) sector of $\epsilon'$
		\item Let $\epsilon$ be an edge of $\sigma$. Then $\sigma$ is on the one-sheeted side of $\epsilon$ if and only if $\epsilon$ is an uppermost edge of $\sigma$.
		\item Let $\epsilon, \epsilon'$ be two lowermost edges of $\sigma$. Then $\sigma$ is on the two-sheeted side of both $\epsilon$ and  $\epsilon'$. Furthermore, $\sigma$ is the right (left) sector of $\epsilon$ if and only if it is the right (left) sector of $\epsilon'$. 
		\item Let $\epsilon$ be a lower edge of $\sigma$ which is  not  lowermost. Let $\epsilon'$ be the edge of $\sigma$ whose terminal vertex is the initial vertex of $\epsilon$. Then $\sigma$ is on the two-sheeted side of both $\epsilon$ and $\epsilon'$. Furthermore, $\sigma$ is the right (left) sector of $\epsilon$ if and only if it is the right (left) sector of $\epsilon'$. 
	\end{enumerate}
\end{lemma}
\begin{proof} \hfill
	\begin{enumerate}
		\item Follows immediately from Figure \ref{fig:veering_branched_surface}.
		\item The backward direction follows immediately from Figure \ref{fig:veering_branched_surface}. Conversely, suppose that $\epsilon$ is not an uppermost edge of $\sigma$. Then there is an edge $\epsilon'$ of~$\sigma$ such that $(\epsilon, \epsilon')$ is a positively oriented turn in the boundary of $\sigma$. Part~(1) then implies that  $\epsilon$ must be on the two-sheeted side of $\sigma$.
		\item Follows immediately from Figure \ref{fig:veering_branched_surface}.
		\item By the assumption, $(\epsilon', \epsilon)$ is a positively oriented turn in the boundary of $\sigma$. Since $\epsilon$ is not an uppermost edge of $\sigma$, the statement follows from parts (1) and (2). \qedhere
	\end{enumerate}
	\end{proof}


Parts (2) and (3) of Lemma \ref{lem:direction:well:defined} together with Definition \ref{def:taut}  imply that a sector~$\sigma$ of a veering branched surface must have at least four edges in its boundary: two lowermost edges and two uppermost edges. Parts (3) and (4) imply that $\sigma$ is either the right sector of all its lower edges, or it is the left sector of all its lower edges. Therefore  we can assign two colors, for instance red and blue, to the sectors of a veering branched surface so that $\sigma$ is red (respectively, blue) if and only if  it is the left  (respectively, right) sector of $\epsilon$ for every lower edge $\epsilon$ of $\sigma$.  Colors on the sectors of $\B$ are  included in Figure \ref{fig:veering_branched_surface}.  The grey sector can be either red or blue; its color cannot be deduced from this local picture. The described coloring of sectors of a veering branched surface is dual to a coloring of edges of a veering triangulation that appears in the
the classical definition of veeringness; see for instance \cite[Definition 5.1]{SchleimSegLinks}.  

	\begin{remark}
	The term \emph{veering branched surface} was introduced by Tsang in~\cite{Tsang-branched-surfaces}. However, the author of \cite{Tsang-branched-surfaces} orients the branch locus of a veering branched surface in the opposite direction.
\end{remark} 

\subsubsection{The Veering Census}\label{sec:census}
Data on veering triangulations of orientable 3-manifolds consisting of up to 16 tetrahedra is available in the Veering Census \cite{VeeringCensus}. A veering triangulation in the census is described by a string of the form
\begin{equation}\label{string}
\texttt{[isoSig]\underline{ }[taut angle structure]}.
\end{equation}
The first part of this string is the isomorphism signature of the triangulation. It identifies a triangulation uniquely up to a combinatorial isomorphism \cite[Section 3]{Burton_isoSig}. The second part of the string records the taut structure, up to reversing orientation on all dual edges. 
A string of the form \eqref{string} is called a \emph{taut signature} and we use it whenever we refer to any particular veering triangulation from the Veering Census.

	\subsection{Vertical surgeries}\label{sec:vertical:surgery}
	In \cite[Section 4]{Tsang-branched-surfaces} Tsang introduced two types of operations that one can perform on a veering branched surface: a \emph{horizontal} and a \emph{vertical surgery}. The crucial property of these operations is that the triangulation dual to the resulting branched surface is veering and has more tetrahedra than the initial veering triangulation  \cite[Propositions 4.4 and 4.8]{Tsang-branched-surfaces}. In this paper we will use vertical surgeries to construct infinitely many veering triangulations with a vanishing taut polynomial. 
	
	Let $\V$ be a veering triangulation of $M$ with a veering branched surface $\B$. Let~$\lambda$ be an oriented simple closed curve smoothly embedded in $\B$ which avoids vertices of~$\B$ and which admits a smooth regular neighborhood $A$ in $\B$ that is an annulus. Then there is a tubular neighborhood $N$ of $\lambda$ in $M$ such that $A$ splits $N$ into two half-neighborhoods $N_1$, $N_2$, each homeomorphic to a solid torus. 
	
	\begin{definition}\label{def:vertical:surgery:curve}\cite[Definition 4.6]{Tsang-branched-surfaces} The curve~$\lambda$ is a \emph{vertical surgery curve} if 
	\begin{enumerate}
		\item for $i=1,2$ all arcs in the branch locus of $N_i \cap \B$ are oriented in the same direction, but the orientations of branch arcs in $N_1 \cap \B$ and $N_2 \cap B$ are opposite,
		\item $\lambda$ is oriented so that for every arc $a$ in the branch locus of $N \cap B$ which is traversed by $\lambda$ the orientation of $\lambda$ points from the one-sheeted side of $a$ to the two-sheeted side of $a$.
	\end{enumerate}
An example of  a vertical surgery curve is presented in Figure \ref{fig:vertical_surgery} (a).
\end{definition}
\begin{figure}[h]
	\begin{center}
		\includegraphics[width=0.4\textwidth]{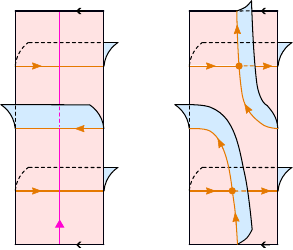} 
		\put(-137, -15){(a)}
			\put(-37, -15){(b)}
			\put(-125, 15){$\lambda$}
	\end{center}
	\caption{(a)  An example of a vertical surgery curve. (b) The result of the $(1,1)$-vertical surgery along~$\lambda$. The dual veering triangulation has two more tetrahedra, one for each new vertex of the surgered branched surface.}
	\label{fig:vertical_surgery}
\end{figure}

Let $\mu$ be the meridian of $N_1$ oriented consistently with the arcs in the branch locus of $N_1 \cap \B$; this orientation is well defined by Definition~\ref{def:vertical:surgery:curve}~(1). Fix an integer $ k \geq 0$. The $(1,k)$-Dehn surgery of $M$ along the core curve of $N_1$ can be realized by  cutting~$N_1$ out of $M$ and gluing it back via a map which sends $\mu$ to $\mu + k\lambda$. In \cite[Section 4.2]{Tsang-branched-surfaces} Tsang observed that this can be performed so that the surgered manifold admits a veering branched surface and thus a veering triangulation; see Figure \ref{fig:vertical_surgery} (b). We call the corresponding operation on the level of branched surfaces the \emph{$(1,k)$-vertical surgery along $\lambda$}. If $\V$ has $n$ tetrahedra, and $n_i$ denotes the number of arcs in the branch locus of  $N_i \cap \B$ then the  veering triangulation dual to the surgered branched surface has $n + k n_1 n_2$ tetrahedra.

\subsubsection{Searching for vertical surgery curves}

In Section \ref{sec:main} we will be looking for a vertical surgery curve with special properties. Thus we need to know more generally how to find vertical surgery curves in veering branched surfaces.  Recall from Lemma~\ref{lem:direction:well:defined}~(1) that if $(\epsilon, \epsilon')$ is a positively oriented turn in the boundary of $\sigma$ then either $\sigma$ is on the one-sheeted side of $\epsilon'$ and on the two-sheeted side of $\epsilon$ or $\sigma$ is on the two-sheeted side of both $\epsilon$ and $\epsilon'$. In the latter case we say that the turn $(\epsilon, \epsilon')$ is \emph{smooth}. 
A reader might wish to consult Figure \ref{fig:veering_branched_surface} to check that a smooth positively oriented turn indeed seems smooth; see also Figure \ref{fig:turns} (b). This terminology will be used in the proof of the following lemma.
\begin{lemma}\label{lem:colors:on:the:1:sheeted:side}
	Let $\sigma$ be a sector of a veering branched surface.
	\begin{enumerate}
		\item Let $\epsilon$ be a lowermost edge of $\sigma$. Then the sector on the one-sheeted side of $\epsilon$ has the same color as $\sigma$.
		\item Let $\epsilon$ be a lower edge of $\sigma$ which is not lowermost. Then the sector on the one-sheeted side of $\epsilon$ has a different color than $\sigma$.
	\end{enumerate}
\end{lemma}
\begin{proof}
	\noindent \begin{enumerate}
		\item Follows immediately from Figure \ref{fig:veering_branched_surface}.
		\item Since $\epsilon$ is not lowermost, there is an edge $\epsilon'$ of $\sigma$ such that $(\epsilon', \epsilon)$ is a positively oriented turn in the boundary of $\sigma$. Since $\epsilon$ is not uppermost, Lemma \ref{lem:direction:well:defined}~(2) implies that  $(\epsilon', \epsilon)$ is smooth. Therefore it follows from Figure \ref{fig:veering_branched_surface} that the sector on the one-sheeted side of~$\epsilon$ is of a different color than $\sigma$. \qedhere
	\end{enumerate}
\end{proof}

Using Lemmas \ref{lem:direction:well:defined} and \ref{lem:colors:on:the:1:sheeted:side} in Figure \ref{fig:nbhd_sector} we present a neighborhood of a red sector of a veering branched surface with seven vertices in the boundary. If there are less vertices then the overall structure of the sector is the same ---  just some of the `flaps' attached to the lower parts of the boundary of the sector are missing. If there are more vertices then then there are more `flaps'. By reflecting Figure \ref{fig:nbhd_sector} across the plane containing the central sector and switching red and blue colors we would obtain a  picture of a neighborhood of a blue sector with seven vertices in its boundary.

\begin{figure}[h]
	\begin{center}
		\includegraphics[width=0.4\textwidth]{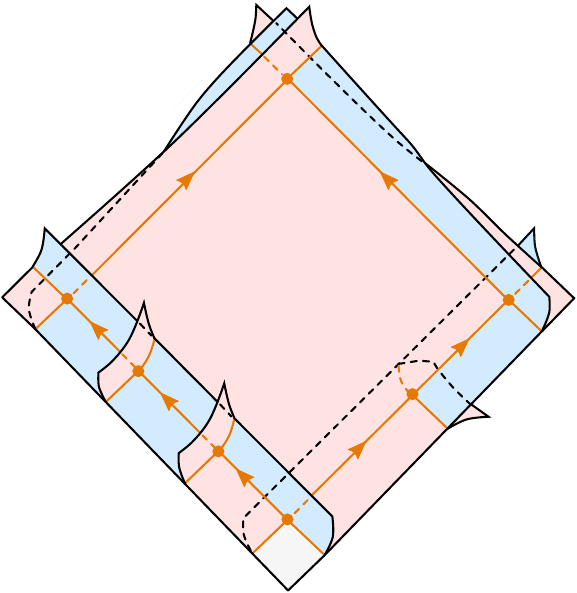} 
	\end{center}
	\caption{A neighborhood of a red sector of a veering branched surface with~seven vertices in its boundary. Orientations on edges of the neighborhing sectors are ommited; they all point upwards.}
	\label{fig:nbhd_sector}
\end{figure}

We are now ready to state a useful characterization of vertical surgery curves.

\begin{lemma}\label{lem:vertical:surgery:curves}
Let $\lambda$ be a vertical surgery curve in a veering branched surface $\B$. Let $\sigma$ be a sector of~$\B$ such that $\lambda \cap \mathrm{int}(\sigma) \neq \emptyset$. A connected component $c$ of $\lambda \cap \sigma$ is an oriented simple arc which connects a lowermost edge of $\sigma$ to an uppermost edge of $\sigma$. 
\end{lemma}
\begin{proof}
	Condition (2) of Definition  \ref{def:vertical:surgery:curve}  implies that if $A$ is a smooth regular neighborhood of $\lambda$ in $\B$ then the sectors of $\B$ attached to $A$ are all combed in the same direction. Thus it follows from Lemma \ref{lem:direction:well:defined} that  $c$ must run from a lower edge of $\sigma$ to one of the uppermost edges of $\sigma$; see also Figure \ref{fig:nbhd_sector}. By condition (1) of Definition  \ref{def:vertical:surgery:curve},
	all sectors whose interiors are traversed by $\lambda$ have the same color. Thus $\lambda$ must enter $\sigma$ from a sector of the same color as $\sigma$. Lemma \ref{lem:colors:on:the:1:sheeted:side} implies that  $c$ runs from one of the lowermost edges of $\sigma$ to one of its uppermost edges.
	\end{proof}
	\subsection{The taut polynomial}\label{sec:taut:poly}

Let $\mathcal{V}$ be a finite veering triangulation of a 3-manifold~$M$, with the set $T$ of tetrahedra, the set $F$ of 2-dimensional faces and the set $E$ of edges. Let
\[H= \bigslant{H_1(M;\zz)}{\text{torsion}}.\] 
We call the cover of $M$ associated to the kernel of the projection 	\[\pi_1(M)\rightarrow H\] the \emph{maximal free abelian cover} of $M$, and denote it by $\Mab$.
Its deck group is isomorphic to $H$. 
The cover $\Mab$ admits a veering triangulation $\Vab$ induced by the covering map $\Mab \rightarrow M$. 
We can identify the sets of tetrahedra, faces, and edges of $\Vab$ with $H \times T$, $H \times F$, and $H \times E$, respectively. We denote the elements of $H \times X$, where $X \in \lbrace T, F, E \rbrace$ by $h\cdot x$, where $h \in H$ and $x \in X$. We say that a simplex $h\cdot x$ of $\V^{fab}$ has the \emph{$H$-coefficient} equal to $h$, or that it is \emph{decorated} with $h$.
We use the multiplicative convention for $H$, and set the action of $H$ on $H\cdot X$ to be given by $h\cdot(g\cdot x) \mapsto (hg)\cdot x$. This induces an identification of the free abelian groups generated by $H\times T$, $H\times F$ and  $H\times E$  with the free $\zz\lbrack H\rbrack$-modules $\zz\lbrack H\rbrack^T, \zz\lbrack H\rbrack^F, \zz\lbrack H\rbrack^E$, respectively.

Let $\B$ be the veering branched surface of $\V$. If $\Vab$ and $\B^{fab}$ are put in a dual position then for every face $h \cdot f \in H \times F$  the intersection of~$\B^{fab}$ with $h\cdot f$ is a 3-valent train track $\tau_{h\cdot f}$ with one switch $v$ in the interior of $h\cdot f$ and three branches: each joining $v$ with the mid-point of an edge of~$h\cdot f$; see Figure~\ref{fig:switch_relation}. We call this train track the \emph{stable train track} of $h\cdot f$. It determines a \emph{switch relation} between its three half-branches: the large half-branch is equal to the sum of the two small half-branches. By identifying the half-branches with the edges in the boundary of $h\cdot f$ which they meet, we obtain a relation between the edges in the boundary of $h\cdot f$. Observe that the switch relation associated of $h' \cdot f$ can be obtained from the switch relation associated to $h\cdot f$  by multiplying each of its terms by $h'h^{-1}$. That is, up to multiplication by elements of $H$ there are only $|F|$ switch relations.

\begin{figure}[h]
	\begin{center}
		\includegraphics[width=0.236\textwidth]{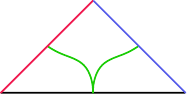} 
		\put(-60,-11){$h_0 \cdot e_0$}
		\put(-111,22){$h_1 \cdot e_1$}
		\put(-15,22){$h_2 \cdot e_2$}
	\end{center}
	\caption{The stable train track in a triangle of $\Vab$ determines the switch relation  $h_0\cdot e_0 = h_1\cdot e_1 + h_2\cdot e_2$ between the edges in its boundary.}
	\label{fig:switch_relation}
\end{figure}

Using the above observations, we  define a $\ZH$-module homomorphism \[D : \ZH^F \rightarrow \ZH^E\] which assigns to $1\cdot f$ 
a rearrangement of the switch relation of $\tau_{1\cdot f}$. For example, the image under $D$ of the triangle presented in Figure \ref{fig:switch_relation} is equal to
\[h_0\cdot e_0 - h_1\cdot e_1 - h_2 \cdot e_2 \in \ZH^E.\] We call $D$ the \emph{full branch relations matrix} of $\V$.  If $\V$ has $n$ tetrahedra then $D$ can be presented as an $n\times 2n$ matrix with $\ZH$ entries. 
\begin{definition}\label{def:taut_poly} \cite[Section 3.1]{LMT}
	The taut polynomial $\Theta_\V$ of $\V$ is the greatest common divisor of the maximal minors of $D$.
\end{definition}

\subsubsection{Standard computation}\label{sec:computation}
According to Definition \ref{def:taut_poly}, if $\V$ has $n$ tetrahedra then to compute $\Theta_\V$ one needs to find the greatest common divisor of $2n \choose n$ minors of a matrix~$D$ with coefficients in $\ZH$.
Thankfully, there are some predictable linear dependencies between the columns of $D$. Thus one does not have to look at all maximal minors of $D$ to compute $\Theta_\V$; see \cite[Lemma~5.4]{Parlak-computation}. More precisely, let $\Gamma$ be the dual graph of $\V$. Choose a spanning tree~$Y$ of $\Gamma$ and let $D_Y$ be the matrix obtained from $D$ by deleting the columns associated to the faces of $\V$ that are dual to the edges of $Y$. We call $D_Y$ the \emph{branch relations matrix} of $\V$ \emph{reduced modulo $Y$} or, when the choice of a tree is evident, \emph{reduced branch relations matrix}.
\begin{fact} \cite[Corollary 5.7]{Parlak-computation} \label{fact:via:tree}
	Let $Y$ be a spanning tree of the dual graph of $\V$. The taut polynomial $\Theta_{\V}$ of $\V$ is equal to the greatest common divisor of the branch relations matrix of $\V$ reduced modulo $Y$.\qed
\end{fact}

Practical computation of $\Theta_\V$ is explained in great detail in \cite[\mbox{Sections 4 -- 5}]{Parlak-computation}. Below we briefly recall the main ideas, as they will be helpful in understanding computations in Section \ref{sec:initial}.

Fix a spanning tree $Y$ of $\Gamma$ and its lift $Y^{fab}$ to the dual graph $\Gamma^{fab}$ of $\V^{fab}$. 
We decorate a tetrahedron of $\Vab$ with $1\in H$ if and only if it is dual to a vertex of $Y^{fab}$.  We decorate a face of $\Vab$ with $1\in H$ if and only if it is a bottom face of a tetrahedron decorated with 1, and we decorate an edge of $\Vab$ with $1\in H$ if and only if it is the bottom diagonal of a tetrahedron decorated with~1. These conventions determine decorations of all other simplices of $\V^{fab}$. Note that edges and faces inherit their $H$-coefficients from the tetrahedron immediately above them. 

By Definition \ref{def:taut_poly}, to compute $\Theta_\V$ it suffices to know for every $f \in F$ the structure of the stable train track in $1\cdot f$ and the $H$-coefficients of the edges in the boundary of $1\cdot f$.  Equivalently, we can decorate sectors, edges, and vertices of $\B^{fab}$ with the $H$-coefficients of their dual edges, faces, and tetrahedra, respectively, and find for each $e \in E$ the $H$-coefficients of the edges in the boundary of the sector $1 \cdot \sigma(e)$ of $\B^{fab}$. Observe that sector $1 \cdot \sigma(e)$ has edge $h^{-1}\cdot \epsilon(f)$ in its boundary if and only if face $1\cdot f$ has edge $h\cdot e$ in its boundary. Furthermore, by Lemma \ref{lem:direction:well:defined} (2), edge $h\cdot e$ is dual to the large half-branch of $1\cdot f$ if and only if the edge  $h^{-1}\cdot \epsilon(f)$  of $\B^{fab}$ is an uppermost edge of $1 \cdot \sigma(e)$; by our labelling conventions this actually implies $h=1$. It follows that we can find the entries of the full branch relations matrix of $\V$ using only sectors $\B^{fab}$ that are decorated with $1\in H$.

The advantage of the dual approach for computing the taut polynomial is that it is very simple to find the $H$-coefficients of all edges in the boundary of $1\cdot \sigma$. 
If an edge of $\B^{fab}$ is a lift of an edge of the spanning tree $Y$ then the $H$-coefficients of its two endpoints are the same (and also equal to the $H$-coefficient of the edge itself). If an edge of $\B^{fab}$ is a lift of an edge $\epsilon$ of $\Gamma$ which is not in $Y$ then there is a unique cycle $\zeta(\epsilon)$ in the subgraph $\epsilon \cup Y$ of $\Gamma$. We call the image of $\zeta(\epsilon)$ under the projection $\pi_1(M) \rightarrow H_1(M;\zz) \rightarrow H$ the \emph{$H$-pairing} of $\epsilon$, or the $H$-pairing of $f$ if $\epsilon = \epsilon(f)$. If $h$ is the $H$-pairing of $\epsilon$ then the $H$-coefficient of the initial vertex of  $1\cdot \epsilon$ is equal to $h^{-1}$. By our labelling  conventions, the uppermost edges of a sector $1 \cdot \sigma$ of $\B^{fab}$ have  their $H$-coefficients equal to 1. The $H$-coefficients of the lower edges of $1\cdot \sigma$ can be found iteratively: the $H$-coefficient of the next  edge in the boundary of $1 \cdot \sigma$ (from the top, in the same side-arc) is equal to the product of the $H$-coefficient and the inverse of the $H$-pairing of the previous edge. 

 Now it remains to explain how to compute the $H$-pairings. We say that a face of~$\V$ is a \emph{tree face} if it is dual to an edge of $Y$. Otherwise we say that $f$ is a \emph{non-tree face}.  Let $\partial: \zz^E \rightarrow \zz^F$ be the cellular boundary map from the sectors of $\B$ to the edges of $\B$. Since $\B$ is a deformation retract of $M$, $H_1(M;\zz)$ is isomorphic to the cokernel of the homomorphism $\partial_Y$ obtained from  $\partial: \zz^E \rightarrow \zz^F$ by deleting the rows corresponding to the edges of~$Y$. This follows from the standard `contraction along a tree' argument. This cokernel can be computed using Smith normal forms. Namely, let
\begin{equation}
\label{eqn:SNF}
S = U \partial_Y V
\end{equation}
be the Smith normal form of $\partial_Y$. If $H$ is of rank $r$ then the last $r$ entries of the $i$-th column of $U$ gives the $H$-pairing of the $i$-th non-tree face of $\V$. The $H$-pairings of tree faces are all trivial. We refer the reader to \cite[Section 4.2]{Parlak-computation} for a more thorough discussion. 

	\subsubsection{Computation via Fox calculus}\label{sec:computation:fox}
An even faster method to compute the taut polynomial relies on the fact that it is equal to a certain twisted Alexander polynomial of the underlying manifold. More precisely, the stable branched surface $\B$ of $\V$ determines the \emph{orientation homomorphism}
\[\omega: \pi_1(M) \rightarrow \lbrace -1, 1 \rbrace \]
such that $\omega(\gamma) = 1$ if and only if $\gamma$ lifts to a loop in the orientable double cover of~$\B$; see \cite[Section 4]{parlak-alex}. Let $\pi: \pi_1(M) \rightarrow H$ be the natural projection arising from the abelianization and killing the torsion of $H_1(M;\zz)$. We can use these two homomorphisms to construct a \emph{tensor representation}
\begin{gather*}
\omega\otimes \pi: \pi_1(M) \rightarrow \mathrm{GL}(1, \ZH)\cong \pm H\\
\gamma \mapsto \omega(\gamma) \cdot \pi(\gamma)
\end{gather*}
and thus a \emph{twisted Alexander polynomial} $\Delta_M^{\omega \otimes \pi}$ of $M$. We refer the reader to \cite{FriedlVidussi_survey} for an overview of twisted Alexander polynomials. Their importance in this paper follows from the following result.
\begin{theorem} \cite[Theorem 5.7]{parlak-alex} \label{thm:taut:is:twisted}
	Let $\V$ be a veering triangulation of $M$. Then
	\[
	\pushQED{\qed} 
	\Theta_\V = \Delta_M^{\omega\otimes \pi}.
	\qedhere
	\popQED\]
\end{theorem}

Thus the taut polynomial of $\V$, like any twisted Alexander polynomial of $M$, can be computed from a  presentation of $\pi_1(M)$ using \emph{Fox calculus} \cite{FoxCalculus1, FoxCalculus5}. We outline this approach below. 

Let $F$ be a free group on $m$ generators $x_1, \ldots, x_m$. For $i =1, 2, \ldots, m$ the \emph{Fox derivative} with respect to $x_i$ is a function
\[\frac{\partial}{\partial x_i}: F \rightarrow \zz \lbrack F \rbrack\] defined by the following three axioms
\begin{gather*}
\frac{\partial}{\partial x_i}(x_j) = \begin{cases}
1 & \text{ if } i=j \\
0 & \text{ if } i\neq j
\end{cases}, \hspace{1.5cm} 
\frac{\partial}{\partial x_i}(1) = 0, \hspace{1.5cm}
\frac{\partial}{\partial x_i}(uv) = \frac{\partial}{\partial x_i}(u) + u \frac{\partial}{\partial x_i}(v)
\end{gather*}
for any $u, v \in F$. It was introduced by Fox in \cite{FoxCalculus1}. If $G$ is a finitely presented group with presentation
\begin{equation}\label{eqn:G:presentation}
G = \langle X \ | \ R \rangle  = \langle x_1, \ldots, x_m \ | \ r_1, \ldots, r_l \rangle \end{equation} then the matrix
\[\frac{\partial R}{\partial X} = \left\lbrack \frac{\partial r_i}{\partial x_j} \right\rbrack_{\substack{i=1, \ldots, m \\ \hspace{-0.18cm}j = 1, \ldots, l}} \] is called the \emph{Jacobian} of the presentation \eqref{eqn:G:presentation}. Let $G^{fab}=  \bigslant{\left(G/[G,G ]\right)}{\text{torsion}}$. Let $\pi: \zz\lbrack G \rbrack \rightarrow \zz \lbrack G^{fab}\rbrack $ be induced by the abelianization and torsion-killing epimorphism.
The matrix $\frac{\partial R}{\partial X}^\pi$ obtained from $\frac{\partial R}{\partial X}$ by mapping its entries through \mbox{$\pi: \zz \lbrack G \rbrack \rightarrow \zz \lbrack G^\pi \rbrack$} is called the \emph{Alexander matrix} of \eqref{eqn:G:presentation}. This terminology is motivated by the fact that the greatest common divisor of the $(m-1) \times (m-1)$ minors of $\frac{\partial R}{\partial X}^\pi$ is equal to the \emph{Alexander polynomial} of $G$ \cite[p. 209]{FoxCalculus2}. 
The greatest common divisor of the $(m-1) \times (m-1)$ minors of the matrix $\frac{\partial R}{\partial X}^{\omega\otimes \pi}$ gives the Alexander polynomial twisted by the representation $\omega\otimes \pi$.


This method of computing the taut polynomial has significant computational advantages. If $\V$ has $n$ tetrahedra then one can easily find a presentation for $\pi_1(M)$ with $n+1$ generators and $n$ relations using a contraction along a spanning tree of~$\Gamma$; we will do this in Section \ref{sec:fundamental:group}. However, this presentation typically can be simplified to a presentation with much less than $n+1$ generators. 
This reduces the dimensions of matrices that one has to work with. 
In the context of this paper we also get an additional benefit of not having to describe in detail the combinatorics of infinitely many veering triangulations to show that the taut polynomials of all of them vanish; we use that in Proposition \ref{prop:taut:vanishes}. On the other hand, the standard computation of the taut polynomial is used in Section \ref{sec:initial}  to first find good candidates for these  infinitely many veering triangulations.

	\section{Sequence of veering triangulations with a vanishing taut polynomial}\label{sec:main}

	In this section we construct an infinite sequence  $(\V_k)_{k\geq 0}$ of veering triangulations such that for every $k \geq 0$ the taut polynomial $\Theta_k$ of $\V_k$ is equal to zero and
	\begin{equation}\label{eqn:limit:of:tetrahedra}\lim_{k\rightarrow \infty} \left(\text{number of tetrahedra of } \V_k\right) = \infty \end{equation} 

	The idea is relatively simple: 
	\begin{enumerate}
		\item Start with a veering triangulation $\V_0$ whose taut polynomial  vanishes.
		\item Carefully choose a vertical surgery curve $\lambda_0$ in the stable branched surface $\B_0$  of~$\V_0$ so that for every $k \geq 1$ the veering triangulation $\V_k$ obtained from $\V_0$ by the $(1,k)$-vertical surgery along $\lambda_0$ has a vanishing taut polynomial. 
\end{enumerate}

By the definition of a vertical surgery, for every $k \geq 0$ there is a positive integer~$m$ such that the veering triangulation $\V_{k+1}$ has $k\cdot m$ more tetrahedra than $\V_{k}$; see Section~\ref{sec:vertical:surgery}. Thus \eqref{eqn:limit:of:tetrahedra} is necessarily satisfied for $(\V_k)_{k\geq0}$ constructed in this way.

A rough idea to find $\lambda_0$  is as follows. 
Suppose that $\V_0$ has $n$ tetrahedra. Denote its underlying manifold by $M_0$ and set $H_0 = H_1(M_0;\zz)/\text{torsion}$. The condition $\Theta_0 = 0$ implies that for every choice of a spanning tree $Y_0$ of the dual graph $\Gamma_0$ of $\V_0$ at most $n-1$ columns of the reduced branch relations matrix $D_{Y_0}$ are linearly  independent over $\zz \lbrack H_0 \rbrack$. 
Denote by $f_1, \ldots, f_{n+1}$ the non-tree faces of $\V_0$. Without loss of generality we may assume that 
\begin{gather}\label{eqn:linear:dependence}
 \sum\limits_{i=1}^{n} q_i \cdot D_{Y_0}(1\cdot f_i) = 0\\ 
 \label{eqn:linear:dependence:2} q_{n+1} \cdot D_{Y_0}(1\cdot f_{n+1}) + \sum\limits_{i=1}^{n-1} q_i' \cdot D_{Y_0}(1\cdot f_i) = 0
\end{gather} for some $q_i, q_i' \in \zz \lbrack H_0\rbrack $ such that $q_n, q_{n+1} \neq 0$. 

Recall from Section \ref{sec:vertical:surgery} that a vertical surgery changes the structure of the branched surface only in a neighborhood of the surgery curve. Therefore we aim to find a vertical surgery curve $\lambda_0$ which does not interfere with \eqref{eqn:linear:dependence} and  \eqref{eqn:linear:dependence:2}. More precisely, we will associate to \eqref{eqn:linear:dependence} and \eqref{eqn:linear:dependence:2} certain graphs $\beta, \beta'$ embedded in $\B_0$, and will search for~$\lambda_0$ that is disjoint from both of them.

The graph $\beta \subset \B_0$ determined by \eqref{eqn:linear:dependence} is defined as follows. Recall from Figure \ref{fig:switch_relation} that if we arrange $\B_0$ and $\V_0$ so that they are in a dual position, for every face $f$ of~$\V_0$ the intersection $f \cap \B_0$ is a trivalent train track $\tau_f$ with a switch at the intersection $f \cap \Gamma_0$ and three half-branches: each joining the switch with the midpoint of an edge in the boundary of $f$. 
The union
\[\beta = \bigcup_{i=1}^n \lbrace \tau_{f_i} \ | \ q_i \neq 0 \rbrace \subset \B_0.\]
can be seen as a graph in $\B_0$ with the vertex set \[\lbrace f_i \cap \Gamma_0 \ | \ q_i \neq 0 \rbrace \cup \lbrace \sigma(e)\cap e \ | \ e \in \partial f_i \text{ with } q_i \neq 0  \rbrace\]
and whose every edge can be identified with a  half-branch of $\tau_{f_i}$ for  $i$ such that $q_i \neq 0$. We call $\beta$ the \emph{branch cycle} associated to \eqref{eqn:linear:dependence}. This terminology is motivated by the fact that the union of $f_i$ with $q_i \neq \emptyset$ can be seen as a 2-cycle with respect to the `boundary map' $D_0(1)$. 
If a vertex $w$ of $\beta$ contained in the interior of a sector $\sigma$ is two-valent we smoothen out the union of the two half-branches that meet at~$w$ into one smooth edge contained in $\sigma$, and remove $w$ from the vertex set of $\beta$.  Two examples of such smoothened branch cycles are presented in Figure \ref{fig:eesr_sectors_graphs}.

	\subsection{The initial veering triangulation and  a `good' vertical surgery curve}\label{sec:initial}
For $\V_0$ we choose a veering triangulation with the taut signature
\[\texttt{iLLAwQcccedfghhhlnhcqeesr\_12001122}. \]
It has 8 tetrahedra, numbered from 0 to 7, that we illustrate in Figure \ref{fig:eesr_tetrahedra}. 
We denote the stable branched surface of~$\V_0$ by $\B_0$ and the dual graph of $\V_0$ by $\Gamma_0$. Faces of $\V_0$ are denoted by $f_i$, for $i=0,1, \ldots, 15$, and edges of $\V_0$ are denoted by $e_j$, for $j=0, 1,\ldots, 7$. For brevity,  the edge of $\B_0$ that is dual to $f_i$  is denoted by~$\epsilon_i$, and the sector of $\B_0$ that is dual to $e_j$ is denoted by $\sigma_j$.
	\begin{figure}[h]
	\begin{center}
		\includegraphics[width=0.8\textwidth]{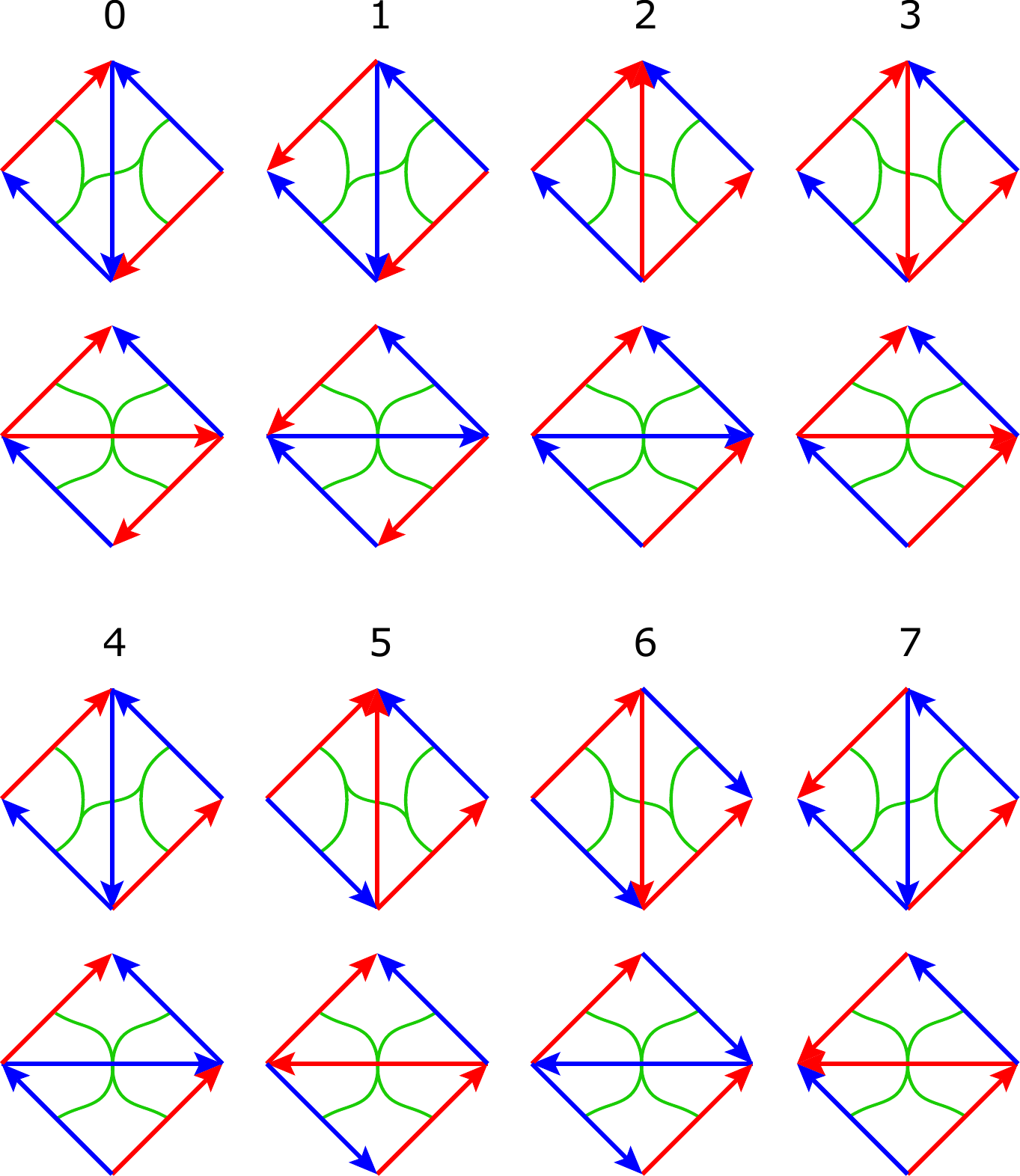} 
		\put(-316,320){$f_3$}
			\put(-276,320){$f_1$}
				\put(-231,320){$f_4$}
					\put(-190,320){$f_5$}
						\put(-146,320){$f_7$}
					\put(-106,320){$f_2$}
					\put(-61,320){$f_0$}
					\put(-21,320){$f_8$}
					\put(-297,255){$f_2$}
					\put(-212,255){$f_6$}
					\put(-126,255){$f_1$}
					\put(-41,255){$f_9$}
					\put(-297,216){$f_0$}
					\put(-212,216){$f_3$}
					\put(-126,216){$f_4$}
					\put(-41,216){$f_7$}
					\put(-318,120){$f_{11}$}
					\put(-277,120){$f_6$}
					\put(-234,120){$f_{12}$}
					\put(-191,120){$f_9$}
					\put(-149,120){$f_{13}$}
					\put(-109,120){$f_{14}$}
					\put(-63,120){$f_{10}$}
					\put(-23,120){$f_{15}$}
					\put(-297,53){$f_5$}
					\put(-213,53){$f_{13}$}
					\put(-128,53){$f_{11}$}
					\put(-43,53){$f_{12}$}
						\put(-299,14){$f_{10}$}
					\put(-211,14){$f_{8}$}
					\put(-128,14){$f_{15}$}
					\put(-43,14){$f_{14}$}
	\end{center}
	\caption{Veering triangulation $\V_0$. Each pair of adjacent triangles represent a pair of top (upper row) or bottom (lower row) faces of a tetrahedron. Each face $f_i$ appears exactly once as a top face of some tetrahedron, and exactly once as a bottom face of some tetrahedron. The identification between these faces is determined by the conditions that their stable train tracks match up and the resulting manifold is orientable.}
	\label{fig:eesr_tetrahedra}
\end{figure}
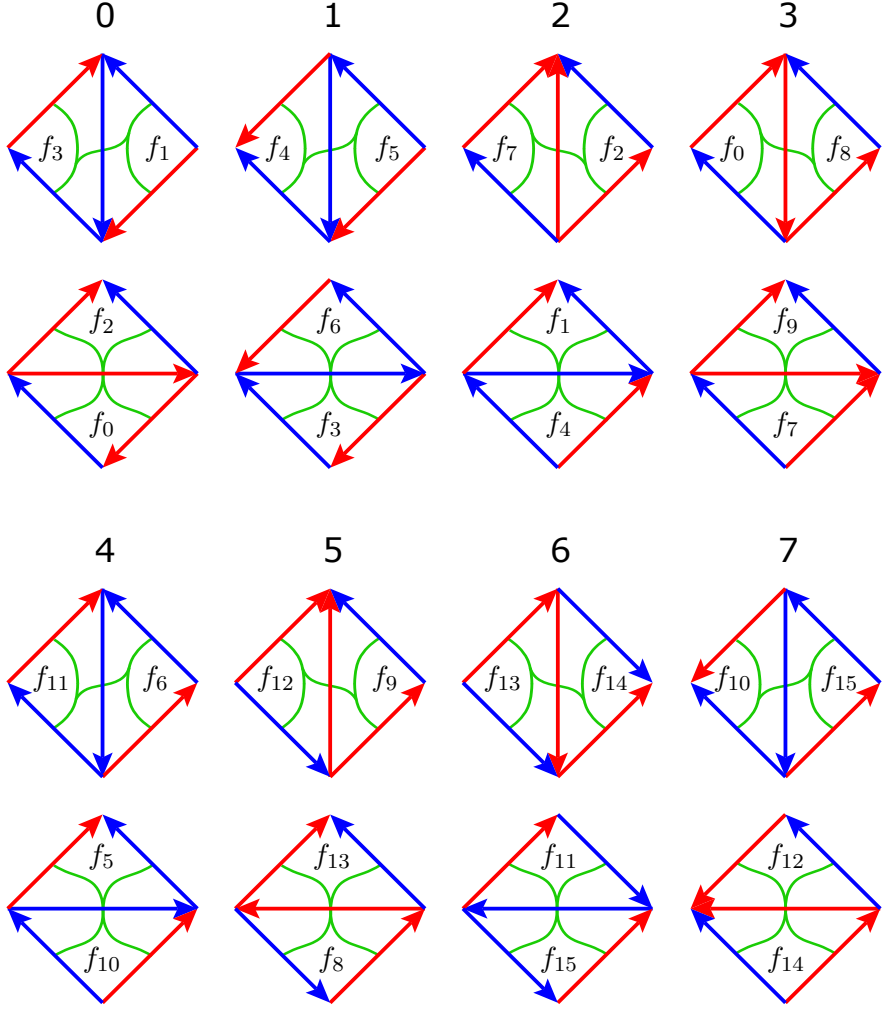

Let $M_0$ be the manifold underlying~$\V_0$. This manifold appears in the SnapPy census~\cite{snappea} under the name \texttt{t12413}. We have $H_1(M_0; \zz) = \zz \oplus \zz/(16\ \zz)$
and thus $H_0 := H_1(M_0;\zz)/\text{torsion} = \zz$. We denote the generator of $H_0$ by $a$. Thus the taut polynomial $\Theta_0$ of $\V_0$ is an element of $\zz\lbrack a^{\pm1}\rbrack$. One can verify that $\Theta_0 = 0$ using tools available at \cite{VeeringGitHub}. We will analyze the branch relations matrix of $\V_0$ in more detail and use it to find a good vertical surgery curve in $\B_0$.

\begin{figure}[h]
	\begin{center}
		\includegraphics[width=0.236\textwidth]{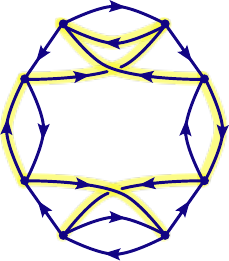} 
		\put(-75,0){\textbf{2}}
		\put(-27,0){\textbf{0}}
		\put(-92,23){\textbf{3}}
		\put(-11,23){\textbf{1}}
		\put(-92,80){\textbf{5}}
		\put(-11,80){\textbf{4}}
		\put(-75,104){\textbf{6}}
		\put(-27,104){\textbf{7}}
		\color{gray}
		\put(-50,-8){1}
		\put(-50,7){2}
		\put(-87,12){7}
		\put(-16,12){3}
		\put(-70,37){0}
		\put(-34,37){4}
		\put(-104,52){8}
		\put(-74,52){9}
		\put(-28,52){5}
		\put(2,52){6}
		\put(-70,65){12}
		\put(-35,65){11}
		\put(-92,92){13}
		\put(-17,92){10}
		\put(-54,113){14}
		\put(-54,96){15}
		\color{black}
	\end{center}
	\caption{The dual graph $\Gamma_0$ of $\V_0$. Edge $\epsilon_i$ is labelled with gray $i$. Vertices are labelled in bold with the number of their dual tetrahedron. The spanning tree $Y_0$ is shaded.}
	\label{fig:eesr_dual_graph}
\end{figure}

The dual graph $\Gamma_0$ of $\V_0$ is presented in Figure \ref{fig:eesr_dual_graph}. As a spanning tree $Y_0$ of $\Gamma_0$ we can choose the subgraph of $\Gamma_0$ consisting of edges  $\epsilon_0, \epsilon_4, \epsilon_6, \epsilon_8, \epsilon_{11}, \epsilon_{12}, \epsilon_{15}$. The $H_0$-pairings of all faces of $\V_0$ are included in Table \ref{tab:H_0-pairings}; see Section \ref{sec:computation} for an explanation of what are $H_0$-pairings and how to compute them. Sectors of $\B_0$ are presented in Figure \ref{fig:eesr_sectors_graphs}. 

\begin{table}[h]
	\begin{tabular}{|>{\columncolor[gray]{0.9}}c|c|c|c|>{\columncolor[gray]{0.9}}c|c|>{\columncolor[gray]{0.9}}c|c|>{\columncolor[gray]{0.9}}c|c|c|>{\columncolor[gray]{0.9}}c|>{\columncolor[gray]{0.9}}c|c|c|>{\columncolor[gray]{0.9}}c|}
		\hline
		$f_0$ & $f_1$ & $f_2$ & $f_3$ & $f_4$ & $f_5$ & $f_6$ & $f_7$& $f_8$ & $f_9$ & $f_{10}$ & $f_{11}$& $f_{12}$ & $f_{13}$ & $f_{14}$ & $f_{15}$ \\ \hline 
		1 & 1 & 1 & 1 & 1 & $a$ & 1 & $a$ & 1 & $a$ & $a$ & 1 & 1 & 1 & 1 & 1\\ \hline
	\end{tabular}
	\vspace{0.3cm}
	\caption{$H_0$-pairings of faces of $\V_0$. Columns associated to tree faces are shaded.}
	\label{tab:H_0-pairings}
\end{table}

%

Let
\[\hspace*{-1.2cm} D_0 = \begin{blockarray}{*{16}{c} l}
\begin{block}{*{16}{>{$\normalsize}c<{$}} l}
\circled{$f_0$} & $f_1$ & $f_2$ & $f_3$ &\circled{$f_4$}&$f_5$&\circled{$f_6$}&$f_7$&\circled{$f_8$}&$f_9$&$f_{10}$&\circled{$f_{11}$}&\circled{$f_{12}$}&$f_{13}$&$f_{14}$& \circled{$f_{15}$}\\
\end{block}
\begin{block}{[cccccccccccccccc]>{$\footnotesize}l<{$}}
-a  &  -a  &   0   &   0  &    0 &     0  &    0  &   1  &    -a &    1   &   0  &   0   &   0     & 0   &   0  &  0\\
1   &   0  &   1   &   0  &   -1  &    0  &   -1  &    0   &   0  &    -1  &   -1   &   0  &    -a   &   0  &    0   &   0\\
0 &0 &-1&-1&1 &0&0&0&0&0&0&0&0&0&0&0\\
-1   &  0 &    0 &   1 &    -a   &  -1  &  1  &  -1  &   0  &   0  &   0  &   0  &   0  &   0  &   0  &   0\\
0  &   0   &  -a  &  -a   &  0  &  -1  &  0 &   -1  &  1  &   0  &   0  &  -1  &   0 &  1  &   0  &   0\\
0  &    0  &    0   &   0   &   0   &  1   &   -a  &    0   &   0  &    0   &  1  &    -a  &    0  &    0  &   -a  &    0\\
0  &    0  &    0  &    0    &  0 &     0  &    0  &    0  &   -1 &    -1 &     -1 &    1  &   -1   &  -a   &   0 & 1 - a \\
0  &    0  &    0   &   0  &    0  &    0  &    0  &    0  &    0   &   0   &   0  &    0 &    1    &  -1  &    0  &   -1\\
\end{block}
\end{blockarray}\]

Following the computation outlined in Section \ref{sec:computation}, using information  from Table \ref{tab:H_0-pairings} and  Figure \ref{fig:eesr_sectors_graphs}, leads to a conclusion that $D_0$ is the full branch relations matrix of $\V_0$.
By Fact \ref{fact:via:tree}, the taut polynomial $\Theta_0$ of $\V_0$  is the greatest common divisor of the maximal minors of the matrix $D_{Y_0}$ obtained from $D_0$ by deleting the columns corresponding to tree faces. For the convenience of the reader, the labels of these columns are encircled.

Since $\Theta_0 = 0$, at most 7 columns of $D_{Y_0}$ are linearly independent over $\zz\lbrack a^{\pm 1} \rbrack$. Indeed, we have
\begin{align}
\label{eqn:f10}
D_0(1\cdot f_{10}) &= D_0(1\cdot f_5) - D_0(1\cdot f_7) + D_0(1\cdot f_9),\\
\label{eqn:f14}
D_0(1\cdot f_{14}) &=  D_0(1\cdot f_1) - a\cdot D_0(1\cdot f_5) + a\cdot D_0(1\cdot f_7).
\end{align}

These linear dependencies determine two branch cycles $\beta, \beta'$ embedded in $\B_0$. We present them in Figure \ref{fig:eesr_sectors_graphs}.
\begin{figure}[h]
	\begin{center}
		\includegraphics[width=0.9\textwidth]{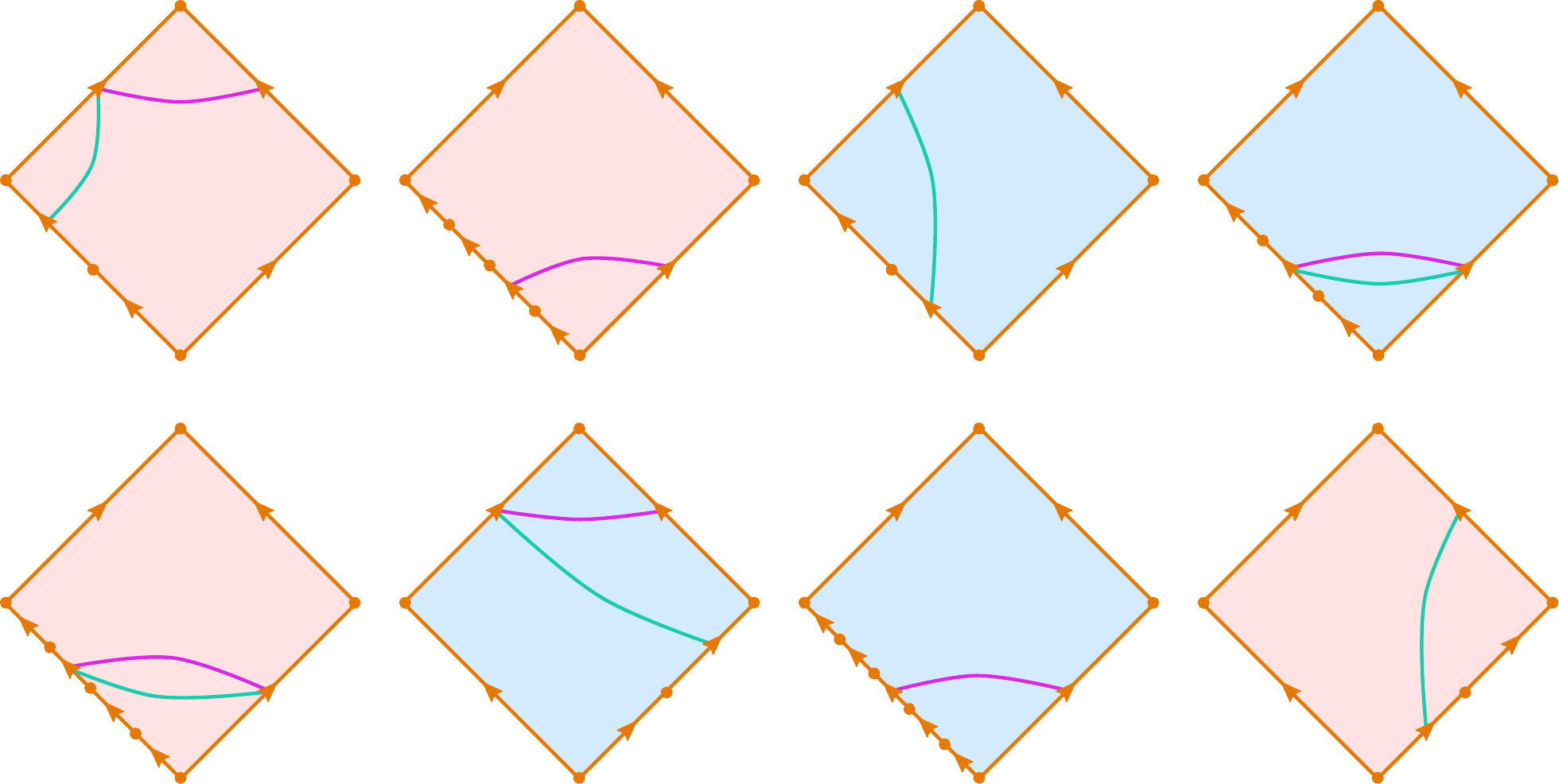} 
				\put(-331,139){$\sigma_0$}
				\put(-235,139){$\sigma_1$}
				\put(-141,139){$\sigma_2$}
				\put(-47,139){$\sigma_3$}
					\put(-331,36){$\sigma_4$}
				\put(-235,36){$\sigma_5$}
				\put(-141,36){$\sigma_6$}
				\put(-47,36){$\sigma_7$}
		\put(-360,165){$\epsilon_7$}
		\put(-302,165){$\epsilon_9$}
		\put(-266,165){$\epsilon_2$}
		\put(-208,165){$\epsilon_0$}
		\put(-170,165){$\epsilon_1$}
		\put(-114,165){$\epsilon_4$}
		\put(-76,165){$\epsilon_3$}
		\put(-19,165){$\epsilon_6$}
		\put(-348,105){$\epsilon_0$}
		\put(-368,125){$\epsilon_1$}
		\put(-301,117){$\epsilon_8$}
		\put(-252,100){$\epsilon_{12}$}
		\put(-262,110){$\epsilon_{10}$}
		\put(-270,120){$\epsilon_{6}$}
		\put(-280,130){$\epsilon_{4}$}
		\put(-207,117){$\epsilon_9$}
		\put(-158,105){$\epsilon_1$}
		\put(-181,127){$\epsilon_2$}
		\put(-112 ,117){$\epsilon_3$}
			\put(-61,102){$\epsilon_4$}
		\put(-75,115){$\epsilon_7$}
		\put(-88,129){$\epsilon_0$}
			\put(-18 ,117){$\epsilon_5$}
			\put(-363,65){$\epsilon_{13}$}
		\put(-302,65){$\epsilon_8$}
		\put(-266,65){$\epsilon_5$}
		\put(-208,65){$\epsilon_{10}$}
		\put(-174,65){$\epsilon_{15}$}
		\put(-114,65){$\epsilon_{11}$}
		\put(-80,65){$\epsilon_{12}$}
		\put(-19,65){$\epsilon_{14}$}
		\put(-343,1){$\epsilon_2$}
		\put(-353,11){$\epsilon_3$}
		\put(-363,21){$\epsilon_5$}
		\put(-376,31){$\epsilon_{11}$}
		\put(-301,17){$\epsilon_7$}
		\put(-264,17){$\epsilon_6$}
			\put(-216,8){$\epsilon_{11}$}
		\put(-197, 27){$\epsilon_{14}$}
			\put(-156,0){$\epsilon_{15}$}
		\put(-165,9){$\epsilon_{13}$}
		\put(-169,17){$\epsilon_9$}
			\put(-178,25){$\epsilon_8$}
		\put(-188,33){$\epsilon_{12}$}
		\put(-113,17){$\epsilon_{10}$}
		\put(-80,17){$\epsilon_{13}$}
			\put(-27,8){$\epsilon_{14}$}
		\put(-8,27){$\epsilon_{15}$}
\put(-30,24){$\gamma$}
	\end{center}
	\caption{Sectors of $\B_0$. Pink: the smoothened branch cycle determined by~\eqref{eqn:f10}. Teal: the smoothened branch cycle determined by \eqref{eqn:f14}. Curve $\gamma$ contained in $\sigma_7$ is the core curve of a M\"obius band whose boundary we choose as the vertical surgery curve.
	}
	\label{fig:eesr_sectors_graphs}
\end{figure}
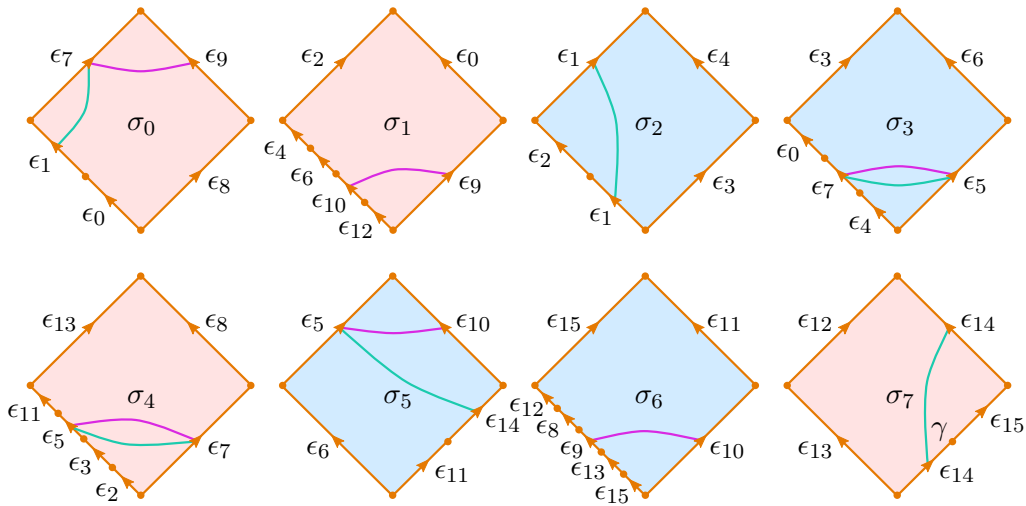

We can now use a characterization of vertical surgery curves (Lemma \ref{lem:vertical:surgery:curves}) to deduce that there are precisely two primitive vertical surgery curves which are disjoint from $\beta \cup \beta'$. Namely, observe that each of the sectors $\sigma_2, \sigma_7$ contains a single edge of $\beta \cup \beta'$ which is the core curve of a M\"obius band embedded in $\B_0$. The boundary of either of these M\"obius bands can be chosen to be our vertical surgery curve $\lambda_0$. We denote by~$\gamma$ the edge of $\beta \cup \beta'$ contained in $\sigma_7$ and set
\[\lambda_0 = \text{the boundary of the M\"obius band with core curve $\gamma$}.\]

 For $k \geq 1$ let $\V_k$ be the veering triangulation obtained from $\V_0$ by peforming the $(1,k)$-vertical surgery along $\lambda_0$. In the remaining sections we will show that for every $k\geq0$ the taut polynomial of $\V_k$ vanishes. Local pictures of the first three surgered veering branched surfaces are illustrated in Figure \ref{fig:surgeries}. The taut signatures of $\V_k$ for $0\leq k \leq 10$ are listed in Table \ref{tab:signatures}. 
 
 \begin{remark}
 	Veering triangulation $\V_0$ carries twice-punctured tori. Cutting $\V_0$ along a carried twice-punctured torus yields a sutured manifold whose guts consist of two solid tori. Each of these solid tori contains one of the two possible vertical surgery curves that we found above.
 \end{remark}
\begin{figure}[h]
	\begin{center}
	\includegraphics[scale=1.2]{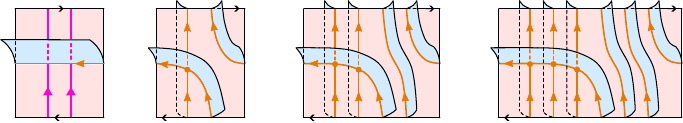}
	\put(-383, 25){$\epsilon_{14}$}
	\put(-350,10){$\lambda_0$}
	\put(-365,-10){$\B_0$}
		\put(-282,-10){$\B_1$}
			\put(-185,-10){$\B_2$}
				\put(-60,-10){$\B_3$}
	\end{center}
	\caption{From the left: a local picture of $\B_0$ near $\lambda_0$, then a local picture of $\B_1$, $\B_2$, $\B_3$, respectively.}
\label{fig:surgeries}
\end{figure}
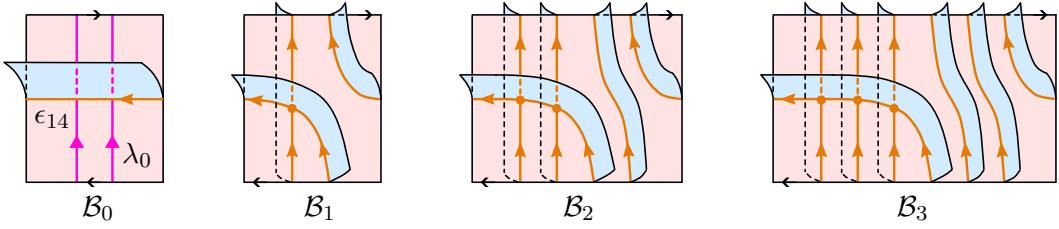
	\begin{table}[h]
		\hspace*{-0.5cm}\begin{tabular}{|c|l|}
			\hline
			$k$ & \multicolumn{1}{c|}{taut signature of $\V_k$} 
			\\ \hline
			0 & \texttt{iLLAwQcccedfghhhlnhcqeesr\_12001122}\\ \hline 
			1& \texttt{jLLAMLQacdefgihiijkaaxxumro\_200210022}\\ \hline
			2& \texttt{kLLAwAAkccedfghhijjlnhcqeeraqj\_1200112202}\\ \hline
			3& \texttt{lLALPzwQcbcbegfhikjkkhhwsaaqqitks\_21122012200}\\ \hline
			4& \texttt{mLAwLAMLQcbbdfhgijlkllhhrhjaaqqitks\_122122012200}\\ \hline
			5& \texttt{nLAMLLMzPkbcbdeghijklmmmxxnxxjaaqqulrj\_0110100210022}\\ \hline
			6& \texttt{oLAMzLPzwQcbcbdefhjiklnmnnhhwhhhkaaxxmlsk\_21121211021100}\\ \hline
			7& \texttt{pLAMzwLAMLQcbbdefgikjlmonoohhrhhhhnaaxxmlsk\_122121211021100}\\ \hline
			8& \texttt{qLAMzMLLMzPkbcbdefghjklmnpoppxxnxxxxxraahhptjr\_0110101011201122}\\ \hline
			9& \texttt{rLAMzMzLPzwQcbcbdefghikmlnoqpqqhhwhhhhhhsaaqqitks\_21121212122012200}\\ \hline
			10& \texttt{sLAMzMzwLAMLQcbbdefghijlnmoprqrrhhrhhhhhhhjaaqqitks\_122121212122012200}\\ \hline
		\end{tabular}
		\vspace{0.1cm}
		\caption{The taut signatures of $\V_k$ for $0\leq k \leq 10$.}
		\label{tab:signatures}
	\end{table}
	\subsection{The strategy}\label{sec:strategy}
	For $k\geq 0$ let $M_k$ be the manifold underlying $\V_k$. Set
	\[H_k: = H_1(M_k;\zz)/\text{torsion}.\]
	Denote by $\pi_k: \pi_1(M_k) \rightarrow H_k$ the abelianization and torsion-killing epimorphism. 
	Let~$\Theta_k$ be the taut polynomial of $\V_k$. To compute $\Theta_k$ we will use  Theorem~\ref{thm:taut:is:twisted}  which says that
	\[\Theta_k = \Delta_{M_k}^{\omega_k\otimes \pi_k},\]
where $\omega_k: \pi_1(M_k) \rightarrow \lbrace -1, 1 \rbrace$ is the orientation homomorphism determined by the stable branched surface $\B_k$ of $\V_k$. As explained in Section \ref{sec:computation:fox}, if $J_k$ is the Jacobian of some presentation of $\pi_1(M_k)$ with $m$ generators then $\Delta_{M_k}^{\omega_k\otimes \pi_k}$ is the greatest common divisor of the $(m-1)\times (m-1)$ minors of $J_k^{\omega_k\otimes \pi_k}$. 

We will express every $M_k$ as a Dehn filling of a single manifold
\[\mathring{M}= M_0 \bez \lambda_0.\] 
Then $J_k$ can be obtained from the Jacobian $\mathring{J}$ of some  presentation of $\pi_1(\mathring{M})$ by adding one row. We will furthermore show that the twisted Alexander matrix $J_k^{\omega_k\otimes \pi_k}$ of $M_k$ can be obtained from a certain twisted Alexander matrix $\mathring{J}^{\mathring{\omega}\otimes \mathring{\pi}_k}$ of $\mathring{M}$ by adding one row. The conclusion that $\Theta_k= 0$ will follow from an observation that $\mathring{J}^{\mathring{\omega}\otimes \mathring{\pi}_k}$ is a zero matrix; see Proposition \ref{prop:taut:vanishes}. 

\subsection{Blowing up $\B_0$ along $\lambda_0$}\label{sec:blowing-up}
Let $A$ be a small smooth regular neighborhood of~$\lambda_0$ which is an annulus. Doubling $A$ produces a tube $T_0$ with two sutures $\lambda, \lambda'$ arising from the boundary components of $A$. There are two connected components of \mbox{$T_0 \bez (\lambda \cup \lambda')$}; we denote them by $A, A'$. We smoothen the obtained 2-complex into a branched surface~$\mathring{\B}$ by requiring that the edges of sutures of $T_0$ have $A$ and $A'$ on their two-sheeted sides. We say that $\mathring{\B}$ is obtained from $\B_0$ by \emph{blowing it up} along~$\lambda_0$. 
A local picture of~$\mathring{\B}$ is presented on the left side of Figure \ref{fig:drilled_eesr}. 
By construction,~$\mathring{\B}$ is a branched surface embedded in~$\mathring{M}$ which is a deformation retract of $\mathring{M}$.

\begin{warning}
	The branched surface $\mathring{\B}$ is not veering. The local picture of $\mathring{\B}$ near the common vertex of edges $\epsilon_{16}, \epsilon_{18}, \epsilon_{19}, \epsilon_{23}$ is neither as in Figure \ref{fig:veering_branched_surface} (a) nor as in Figure~\ref{fig:veering_branched_surface}~(b). Similarly for the common vertex of edges $\epsilon_{17}$, $\epsilon_{20}$, $\epsilon_{21}$, $\epsilon_{22}$.
\end{warning}

	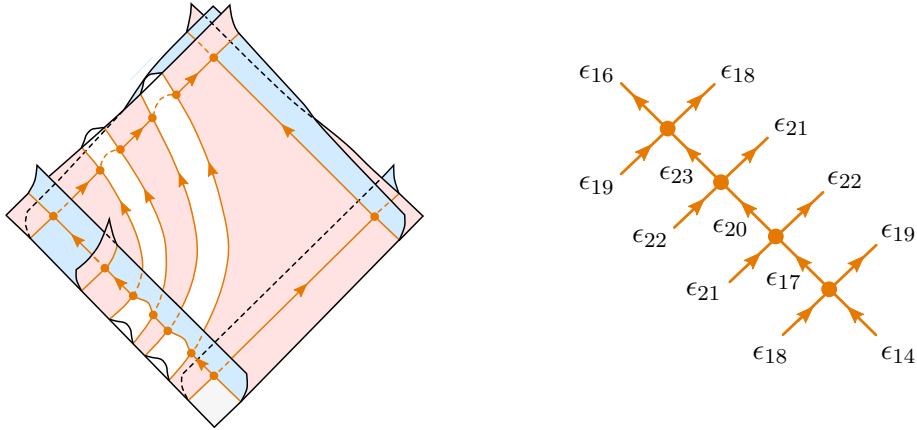
\begin{figure}[h]
	\begin{center}
		\begin{minipage}{0.48\textwidth}
			\centering
			\includegraphics[width=0.8\textwidth]{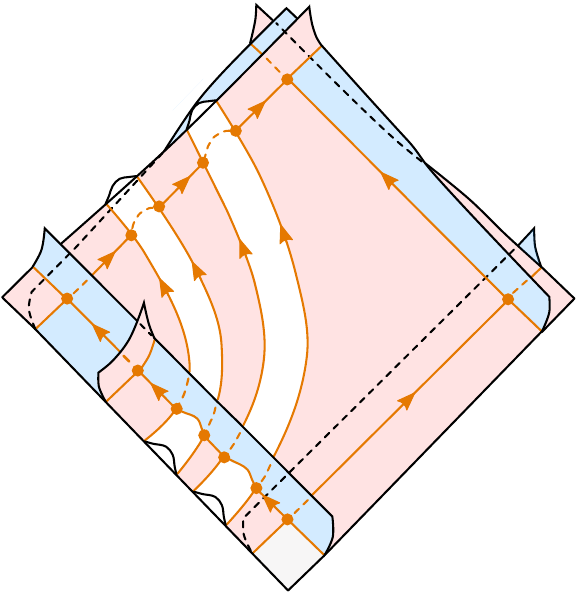} 
		\end{minipage}
		\begin{minipage}{0.48\textwidth}
			\centering
			\includegraphics[width=0.5\textwidth]{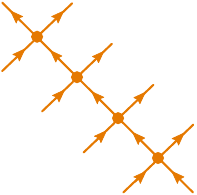}		
			\put(-48,-8){$\epsilon_{18}$}
			\put(0,-8){$\epsilon_{14}$}
			\put(0,38){$\epsilon_{19}$}
			\put(-20,58){$\epsilon_{22}$}
			\put(-40,78){$\epsilon_{21}$}
			\put(-60,98){$\epsilon_{18}$}
			\put(-113,98){$\epsilon_{16}$}
			\put(-113,56){$\epsilon_{19}$}
			\put(-93,36){$\epsilon_{22}$}
			\put(-73,16){$\epsilon_{21}$}
			\put(-43,20){$\epsilon_{17}$}
			\put(-63,40){$\epsilon_{20}$}
			\put(-83,60){$\epsilon_{23}$}
		\end{minipage}
	\end{center}
	\caption{Left: The result of blowing up $\B_0$ along $\lambda_0$. Orientations of some edges are omitted; they all point upwards. Right: Labels on the new edges of~$\mathring{\B}$. The remaining edges of $\mathring{\B}$ inherit their labels from $\B_0$.}
	\label{fig:drilled_eesr}
\end{figure}

The branched surface $\mathring{\B}$ has two applications. First, in Section \ref{sec:fundamental:group} we will use it to find a presentation for $\pi_1(\mathring{M})$. Furthermore, associated to $\mathring{\B}$ there is an orientation homomorphism $\mathring{\omega}: \pi_1(\mathring{M}) \rightarrow \lbrace -1, 1 \rbrace$. In Lemma \ref{lem:omega} we show that $\mathring{\omega}$ factors through the orientation homomorphism $\omega_k:\pi_1(M_k) \rightarrow \lbrace -1, 1\rbrace$ of $\B_k$ for every $k \geq 0$. This fact is later used in Proposition \ref{prop:taut:vanishes}.

\subsection{The fundamental group of $M$}\label{sec:fundamental:group}
Since $\mathring{\B}$ is a deformation retract of $\mathring{M}$, we can use it to find a presentation for $\pi_1(\mathring{M})$. Observe that $\mathring{\B}$ differs from $\B_0$ only in a small neighborhood of $\lambda_0$. Since  $\epsilon_{14}$ is the only edge of $\B_0$ that is traversed by $\lambda_0$, all new vertices of $\mathring{\B}$ appear along  $\epsilon_{14}$. On the right side of Figure \ref{fig:drilled_eesr} we included labels and orientations of all edges of~$\mathring{\B}$ that are adjacent to the new vertices of $\mathring{\B}$. With that notation, the edge $\epsilon_{14}$ of $\B_0$ breaks down into an edge-path $\epsilon_{14}\epsilon_{17}\epsilon_{20}\epsilon_{23}\epsilon_{16}$ in the 1-skeleton of $\mathring{\B}$. The remaining new edges of $\mathring{\B}$ form the sutures of $T_0$.

	The only sectors of $\B_0$ that are affected by the blow-up  are those in which $\epsilon_{14}$ is embedded. Thus for $i\in \lbrace0, 1, 2, 3, 4, 6\rbrace $  the branched surface $ \mathring{\B}$ has a sector  which looks exactly like the sector $\sigma_{i}$ of $\B_0$ presented in Figure \ref{fig:eesr_sectors_graphs}.  Furthermore, $\mathring{\B}$ has a sector obtained from the sector $\sigma_5$  of $\B_0$ by replacing its edge $\epsilon_{14}$ with an edge-path $\epsilon_{14}\epsilon_{17}\epsilon_{20}\epsilon_{23}\epsilon_{16}$. The sector $\sigma_7$ of $\B_0$ breaks down into three sectors of~$\mathring{\B}$; see Figure~\ref{fig:drilled_eesr}. Finally, $\mathring{\B}$ has two more sectors: one contained in  $A \subset T_0$, and one contained in $A' \subset T_0$.
	
		Let $\mathring{\Gamma}$ be the 1-skeleton of $\mathring{\B}$. Recall from Section \ref{sec:initial} that edges $\epsilon_0$, $\epsilon_4$, $\epsilon_6$, $\epsilon_8$, $\epsilon_{11}$, $\epsilon_{12}$ and $\epsilon_{15}$ formed a spanning tree $Y_0$ of $\Gamma_0$. All these edges survive the blow-up of $\B_0$ to $\mathring{\B}$, and thus we can see $Y_0$ as a subtree of $\mathring{\Gamma}$. We complete it to a spanning tree $\mathring{Y}$ of $\mathring{\Gamma}$ by adding edges $\epsilon_{14}$, $\epsilon_{17}$, $\epsilon_{20}$ and $\epsilon_{23}$; see Figure \ref{fig:drilled_dual_graph}. 
		
		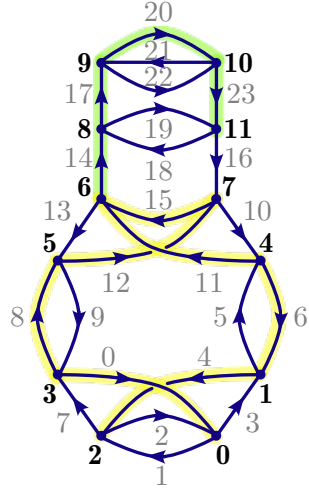
\begin{figure}[h]
			\begin{center}
				\includegraphics[width=0.236\textwidth]{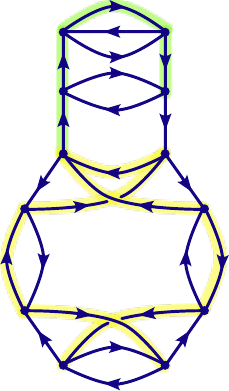} 
				\put(-75,0){\textbf{2}}
				\put(-27,0){\textbf{0}}
				\put(-92,23){\textbf{3}}
				\put(-11,23){\textbf{1}}
				\put(-92,80){\textbf{5}}
				\put(-11,80){\textbf{4}}
				\put(-79,101){\textbf{6}}
				\put(-25,101){\textbf{7}}
				\put(-79,123){\textbf{8}}
			\put(-24,123){\textbf{11}}
				\put(-79,148){\textbf{9}}
			\put(-24,148){\textbf{10}}
				\color{gray}
				\put(-50,-8){1}
				\put(-50,7){2}
				\put(-87,12){7}
				\put(-16,12){3}
				\put(-70,37){0}
				\put(-34,37){4}
				\put(-104,52){8}
				\put(-74,52){9}
				\put(-28,52){5}
				\put(2,52){6}
				\put(-70,65){12}
				\put(-35,65){11}
				\put(-92,92){13}
				\put(-17,92){10}
				\put(-84,112){14}
				\put(-54,96){15}
				\put(-24,112){16}
				\put(-84,136){17}
				\put(-54,107){18}
					\put(-54,123){19}
					\put(-54,167){20}
				\put(-54,152.5){21}
					\put(-54,142.5){22}
				\put(-23,136){23}
				\color{black}
			\end{center}
			\caption{The 1-skeleton $\mathring{\Gamma}$ of $\mathring{\B}$. Edge $\epsilon_i$ is labelled with gray $i$. Vertices are labelled in bold by the number of their dual tetrahedron. The tree~$Y_0$ is shaded yellow. The edges completing $Y_0$ to a spanning tree $\mathring{Y}$ of $\mathring{\Gamma}$ are shaded green.}
			\label{fig:drilled_dual_graph}
		\end{figure}
		
		For every non-tree edge $\epsilon_i$ let $\zeta_i$ be the unique cycle in $\mathring{Y}\cup \epsilon_i$. Then $\pi_1(\mathring{\B})$ is generated by   $\zeta_1, \zeta_2, \zeta_3, \zeta_5,\zeta_7,  \zeta_9, \zeta_{10},   \zeta_{13},   \zeta_{16},  \zeta_{18}, \zeta_{19},  \zeta_{21}, \zeta_{22} $ satisfying the following relations
		\begin{multicols}{3}
\begin{center}
$\zeta_7^{-1}\zeta_1^{-1}\zeta_9,$ \\
$\zeta_2^{-1}\zeta_{10}^{-1}\zeta_9, $\\
$\zeta_1^{-1}\zeta_2^{-1}\zeta_1^{-1}\zeta_3, $\\
$\zeta_3^{-1}\zeta_7^{-1}\zeta_5, $
\end{center}

\begin{center}
$\zeta_{13}^{-1}\zeta_5^{-1}\zeta_3^{-1}\zeta_2^{-1}\zeta_7, $\\
$\zeta_5^{-1}\zeta_{16}\zeta_{10},$\\
$\zeta_9^{-1}\zeta_{13}^{-1}\zeta_{10}, $\\
$\zeta_{13}^{-1}\zeta_{19}\zeta_{16}, $
\end{center}

\begin{center}
$\zeta_{16}^{-1}\zeta_{18}, $\\
$\zeta_{18}^{-1}\zeta_{19}^{-1}\zeta_{22}\zeta_{21}, $\\
$\zeta_{19}^{-1}\zeta_{18}^{-1}\zeta_{21}\zeta_{22},$\\ 
$\zeta_{21}^{-1}\zeta_{22}. $
\end{center}
		\end{multicols}
	
	We used SageMath \cite{sagemath} to simplify this presentation into the following one
	\begin{equation}\label{eqn:drilled:presentation}
	\left\langle 
	\zeta_1, \zeta_{10}, \zeta_{21} \middle\vert\,
	\begin{array}{c} \zeta_{21}^{-2}\zeta_1^{-1}\left(\zeta_{10}^{-1}\zeta_{21}^2\zeta_{10}\zeta_1^{-1}\zeta_{10}^{-1}\zeta_{21}^2\right)^2\zeta_{10}\zeta_1^{-1}\zeta_{21}^{-2}\zeta_{10}, \\ \zeta_{21}^{-2}\left(\zeta_{10}\zeta_1^{-1}\zeta_{10}^{-1}\zeta_{21}^2\right)^2\zeta_1\zeta_{21}^2\zeta_1^{-1}\left(\zeta_{21}^{-2}\zeta_{10}\zeta_1\zeta_{10}^{-1}\right)^2 
	\end{array}
	\right\rangle.
	\end{equation}
	
The simplification isomorphism is given by
\begin{align}
&\zeta_1 \mapsto \zeta_1 &&\zeta_{13} \mapsto \zeta_{21}^2 \nonumber\\
&\zeta_2 \mapsto \zeta_{10}^{-1}\zeta_{21}^{-2}\zeta_{10} && \zeta_{16} \mapsto \zeta_1^{-1}\left(\zeta_{21}^{-2}\zeta_{10}\zeta_1\zeta_{10}^{-}\right)^2 \nonumber\\
&\zeta_3 \mapsto \zeta_1\zeta_{10}^{-1}\zeta_{21}^{-2}\zeta_{10}\zeta_1 && \zeta_{18} \mapsto \zeta_1^{-1}\left(\zeta_{21}^{-2}\zeta_{10}\zeta_1\zeta_{10}^{-1}\right)^2\nonumber\\
&\zeta_5 \mapsto \zeta_1^{-1}\zeta_{21}^{-2}\zeta_{10}\zeta_1\zeta_{10}^{-1}\zeta_{21}^{-2}\zeta_{10}\zeta_1 &&\zeta_{19} \mapsto\left(\zeta_{10}\zeta_1^{-1}\zeta_{10}^{-1}\zeta_{21}^2\right)^2\zeta_1\zeta_{21}^2 \label{eqn:simplification}\\
&\zeta_7 \mapsto \zeta_1^{-1}\zeta_{21}^{-2}\zeta_{10}&&\zeta_{21} \mapsto \zeta_{21}\nonumber\\
&\zeta_9 \mapsto \zeta_{21}^{-2}\zeta_{10} &&\zeta_{22} \mapsto \zeta_{21}\nonumber\\
&\zeta_{10} \mapsto \zeta_{10}& &\nonumber
\end{align}

%
Since the relations in \eqref{eqn:drilled:presentation} are given by relatively long words, the Jacobian of \eqref{eqn:drilled:presentation} has very long entries. For this reason, we list them in Appendix \ref{sec:AppendixA} instead of in the main text. Note that the Jacobian can be found using \texttt{alexander\_matrix()} function in SageMath \cite{sagemath}.

\subsection{Projection $\pi_1(\mathring{M}) \rightarrow H_k$}\label{sec:projection:to:H_k}
Recall from Section \ref{sec:strategy} that $H_k$ denotes \linebreak $H_1(M_k;\zz)/\text{torsion}$ and $\pi_k: \pi_1(M_k) \rightarrow H_k$ is the abelianization and torsion-killing projection. Let \mbox{$\psi_k: \pi_1(\mathring{M}) \rightarrow \pi_1(M_k)$}  be the epimorphism induced by Dehn filling $\mathring{M}$ to $M_k$. Set
\begin{equation}\label{eqn:projection:to:Hk:def}
\mathring{\pi}_k: = \pi_k \circ \psi_k: \pi_1(\mathring{M}) \rightarrow H_k.
\end{equation}

\begin{lemma} \label{lem:projection:to:H_k}
	Let $k\geq 0$. 
	\begin{enumerate}
	\item $H_k \cong \zz = \langle c \ |  \ - \rangle$. 
	\item $\mathring{\pi}_k(\zeta_1) = 1$, $\mathring{\pi}_k(\zeta_{10}) = c$, $\mathring{\pi}_k(\zeta_{21}) = 1$.
	\end{enumerate}
	In particular, $\mathring{\pi}_k$ does not depend on $k$.
\end{lemma}
\begin{proof}
From \eqref{eqn:drilled:presentation} it follows that $H_1(\mathring{M};\zz)$ is isomorphic to the cokernel of the matrix
\begin{equation}\label{eqn:homology:computation}
\begin{bmatrix}
-4 & 0 \\
0 & 0\\
4 & 0
\end{bmatrix}.\end{equation}
Since 
\begin{equation}\label{eqn:SNF}\begin{bmatrix}
4 & 0 \\
0 & 0 \\
0& 0
\end{bmatrix} = 
\begin{bmatrix}
0 & 0 & -1 \\
1 & 0 & 1\\
0 & 1 & 0
\end{bmatrix} \cdot \begin{bmatrix}
-4 & 0 \\
0 & 0\\
4 & 0
\end{bmatrix} \cdot \begin{bmatrix}
-1 & 0\\
0 & 1
\end{bmatrix}\end{equation}
is the Smith normal form of \eqref{eqn:homology:computation}, we get that \[H_1(\mathring{M};\zz) = \zz/4 \oplus \zz \oplus \zz = \langle a, b, c \ | \ 4a = 0\rangle.\] Note that here we use the additive convention for $H_1(\mathring{M};\zz)$, and that all generators $a, b, c$ commute which is not explicitly written as relations for brevity. The equality \eqref{eqn:SNF} implies that
\begin{equation}\label{eqn:homology:classes}
\lbrack\zeta_1\rbrack = b, \hspace{1cm}\lbrack \zeta_{10} \rbrack = c, \hspace{1cm}\lbrack\zeta_{21}\rbrack = -a + b\end{equation} in $H_1(\mathring{M};\zz)$. 
Let $s_k$ be a 1-cycle in $\mathring{\B}$ homotopic to the Dehn filling slope which creates $M_k$ out of~$\mathring{M}$. To find $H_k$ it suffices to find the homology class of $s_k$ expressed in terms of $a, b, c$, and then quotient $\langle a, b, c \ | \ 4a = 0\rangle$ by that expression.
One can consult Figure~\ref{fig:drilled_eesr} to verify that if
\begin{align*}
\mu &= \epsilon_{21}\epsilon_{17}^{-1}\epsilon_{18}^{-1}\epsilon_{23}^{-1},\\
\lambda &= \epsilon_{18}\epsilon_{19}
\end{align*}
then $s_k = \mu\lambda^k$. Using the simplification isomorphism \eqref{eqn:simplification} and the fact that edges $\epsilon_{17}$, $\epsilon_{23}$ are in the tree we deduce that the homology classes of $\mu$ and $\lambda$ are as follows
\begin{align*}
\lbrack \mu \rbrack &= \lbrack \zeta_{21} \rbrack - \lbrack\zeta_{18}\rbrack = 5\lbrack \zeta_{21} \rbrack- \lbrack \zeta_1 \rbrack  =  4b-a,\\
\lbrack \lambda \rbrack &= \lbrack\zeta_{18} \rbrack + \lbrack \zeta_{19} \rbrack =2\lbrack \zeta_{21} \rbrack  = 2b - 2a.
\end{align*}
Therefore \[\lbrack \mu\lambda^k \rbrack = -(2k+1)a + (2k+4)b = \begin{cases}
a + (2k+4)b &\text{if $k$ is odd}\\
-a+(2k+4)b &\text{if $k$ is even}.
\end{cases}\]
Consequently 
\begin{align*}H_1(M_k;\zz) =& \langle a, b, c \ | \ 4a = 0, \pm a+(2k+4)b =0\rangle= \\
= &\langle a, b, c \ | \ 4a=0, \pm a + (2k+4)b = 0, 4(2k+4)b =0 \rangle = \\
= &\langle a, b, c \ | \ \pm a + (2k+4)b = 0, 4(2k+4)b =0 \rangle =\\
= &\langle \pm a+ (2k+4)b, b, c \ | \ \pm a + (2k+4)b = 0, 4(2k+4)b =0 \rangle =\\
=& \langle b, c\ | \ (8k+16)b = 0 \rangle \cong \zz/(8k+16)\zz \oplus \zz.\end{align*}
Hence $H_k \cong \zz = \langle c \ | \ - \rangle$ which proves (1). Part (2) then follows from  \eqref{eqn:homology:classes}.\end{proof}

\subsection{The orientation homomorphism}\label{sec:orientation:homo}
In this section we find the orientation homomorphism $\mathring{\omega}: \pi_1(\mathring{M}) \rightarrow \lbrace -1, 1\rbrace$ of the blown-up branched surface $\mathring{\B}$.  First, we state its relationship to the orientation homomorphisms $\omega_k: \pi_1(M_k) \rightarrow \lbrace -1, 1\rbrace$ of the veering branched surfaces~$\B_k$.

\begin{lemma}\label{lem:omega}
	Let $\mathring{\omega}, \omega_k$ be the orientation homomorphism of $\mathring{\B}$, $\B_k$ respectively. Let $\psi_k: \pi_1(\mathring{M}) \rightarrow \pi_1(M_k)$ be induced by Dehn filling $M$ to $M_k$. Then the following diagram commutes. 
	\[	\begin{tikzcd}
	\pi_1(M) \arrow[d, "\psi_k",swap] \arrow[r, "\mathring{\omega}"] & \lbrace-1, 1 \rbrace \\
	\pi_1(M_k) \arrow[ru, "\omega_k", swap]& 
	\end{tikzcd}\]
\end{lemma}
\begin{proof}
	The fundamental group $\pi_1(M_k)$ is generated by 1-cycles of $\B_k$. For any such cycle $c_k$ there is a cycle $c$ in the 1-skeleton $\mathring{\Gamma}$ of $\mathring{\B}$ such that $\psi_k(c) = c_k $. Since $\mathring{\omega}(\lambda) = \mathring{\omega}(\mu) = 1$ we have $\mathrm{ker} \ \psi_k \leq \mathrm{ker} \ \mathring{\omega}$ and thus  $\mathring{\omega}( c ) = \omega_k( c_k )$. 
\end{proof}

Recall from Section \ref{sec:fundamental:group} that a generator $\zeta_i$ of $\pi_1(\mathring{M})$ is the unique 1-cycle in $\mathring{Y} \cup \epsilon_i$. Using Figure \ref{fig:drilled_dual_graph} we find that 
\[\zeta_1 = \epsilon_1\epsilon_4^{-1}\epsilon_6^{-1}\epsilon_{11}\epsilon_{15}^{-1}\epsilon_{12}^{-1}\epsilon_8^{-1}\epsilon_0,\hspace{1cm} \zeta_{10} = \epsilon_{10}\epsilon_{11}\epsilon_{15}^{-1}, \hspace{1cm} \zeta_{21} = \epsilon_{21}\epsilon_{20} .\]


The first two 1-cycles are disjoint  from $T_0$ and thus can be seen as 1-cycles in  $\B_0$, or, in fact, as 1-cycles in $\B_k$ for any $k\geq 0$. To find their images under $\omega_k$ observe that 1-cycles decompose into turns, and turns in a veering branched surface can be classified into four types presented in Figure \ref{fig:turns}. We define a \emph{sign} of a turn as follows
\begin{gather*}
\mathrm{sgn}\left(\epsilon_{i_1}^{\delta_{i_1}}, \epsilon_{i_2}^{\delta_{i_2}}\right) = \begin{cases}
-1 & \text{ if} \left(\epsilon_{i_1}^{\delta_{i_1}}, \epsilon_{i_2}^{\delta_{i_2}}\right) \text{ is of type illustrated in Figure \ref{fig:turns} (a)}\\
+1 &\text{ otherwise.}
\end{cases}\end{gather*}

	\begin{figure}[h]
	\begin{center}
			\includegraphics[width=0.8\textwidth]{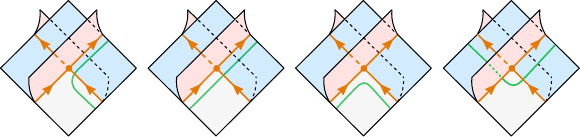} 
			\put(-312,-10){sgn = $-1$}
			\put(-295,-30){(a)}
			\put(-225,-10){sgn = $1$}
			\put(-212,-30){(b)}
			\put(-142,-10){sgn = $1$}
			\put(-127,-30){(c)}
				\put(-58,-10){sgn = $1$}
			\put(-46,-30){(d)}
	\end{center}
	\caption{Turns in a veering branched surface near a bottom vertex of a blue sector. Reflecting the figure across the plane of the sheet (and switching the colors of sectors) yields a picture of turns in a veering branched surface near a bottom vertex of a red sector.}
	\label{fig:turns}
\end{figure}
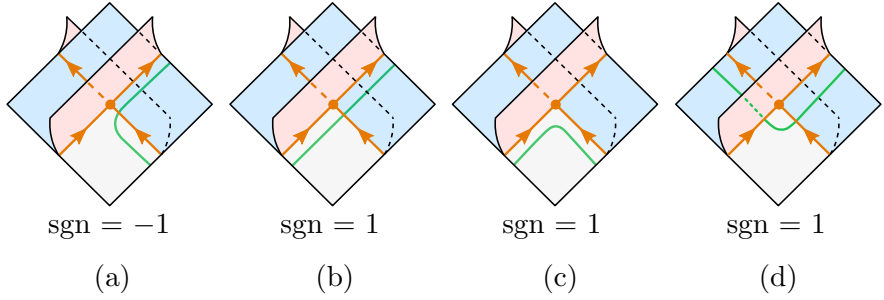

\begin{lemma}\label{lem:omega:turns}
	Suppose that \[\zeta = \epsilon_{i_1}^{\delta_{i_1}} \cdots \epsilon_{i_m}^{\delta_{i_m}}\]
	is a 1-cycle in a veering branched surface $\B$. If $\omega$ is the orientation homomorphism determined by $\B$ then
	\[\omega(\zeta) = \prod_{j=1}^m \mathrm{sgn}\left(\epsilon_{i_j}^{\delta_{i_j}}, \epsilon_{i_{j+1}}^{\delta_{i_{j+1}}}\right),\]
	where the subscript $j$ is taken modulo $m$.
	\end{lemma}
\begin{proof}
	The proof is completely analogous to \cite[Lemma 5.6]{LMT}; the only difference here is that we allow cycles to traverse edges in the direction opposite to their orientation. 
	For every point $x$ in the interior of $\epsilon_{i_j}$ we set a local orientation at $x$ to be determined by a pair (direction in which $\zeta$ passes through $\epsilon_{i_j}$, vector which points towards the one-sheeted side of $\epsilon_{i_j}$). This orientation flips only at a vertex of a turn of type illustrated in Figure \ref{fig:turns} (a). Since these are the only turns with a negative sign, it follows that the image of $\zeta$ under $\omega$ is equal to the product of signs of its turns.
\end{proof}
Now we will use Lemmas \ref{lem:omega} and \ref{lem:omega:turns}  to find the images of $\zeta_1, \zeta_{10}, \zeta_{21}$ under the orientation homomorphism of the blown-up (non-veering) branched surface $\mathring{\B}$. 
\begin{lemma}\label{lem:omega:final}
	Let $\mathring{\omega}: \pi_1(\mathring{M})\rightarrow \lbrace -1, 1\rbrace$ be the orientation homomorphism of $\mathring{\B}$.
	Consider the presentation \eqref{eqn:drilled:presentation} for $\pi_1(\mathring{M})$. The generators $\zeta_1, \zeta_{10}, \zeta_{21}$ of $\pi_1(\mathring{M})$ satisfy
	\[\mathring{\omega}(\zeta_1) = \mathring{\omega}(\zeta_{10}) = \mathring{\omega}(\zeta_{21}) = -1.\]
\end{lemma}
\begin{proof}
The embedding of the 1-cycle $\zeta_{21} = \epsilon_{21}\epsilon_{20}$ into $\mathring{\B}$  is visible in  Figure \ref{fig:drilled_eesr}. It is homotopic to the core of a M\"obius band with boundary $\epsilon_{21}\epsilon_{22}$ embedded in $\mathring{\B}$. Thus clearly $\mathring{\omega}(\zeta_{21}) = -1$.

The image of a 1-cycle $\zeta$ in $\mathring{\B}$ under the orientation homomorphism of $\mathring{\B}$ depends entirely on a small neighborhood of the embedding of $\zeta$ into $\mathring{\B}$. The cycles $\zeta_1 = \epsilon_1\epsilon_4^{-1}\epsilon_6^{-1}\epsilon_{11}\epsilon_{15}^{-1}\epsilon_{12}^{-1}\epsilon_8^{-1}\epsilon_0$, $\zeta_{10} = \epsilon_{10}\epsilon_{11}\epsilon_{15}^{-1}$ are disjoint from the tube $T_0 \subset \mathring{\B}$, and thus they can be realized as 1-cycles in $\B_k$ for any $k\geq 0$. Let $\zeta_i^{(k)}$ be the 1-cycle in $\B_k$ corresponding to $\zeta_i$, for $i=1, 10$. Since $\zeta_i$ and $\zeta_i^{(k)}$ admit homeomorphic neighborhoods in the respective branched surfaces in which they are embedded, we have
\[\mathring{\omega}(\zeta_i)= \omega_k\left(\zeta_i^{(k)}\right). \]
Thus it suffices to find~$\omega_0\left(\zeta_i^{(0)}\right)$. Since~$\B_0$ is veering, we can use Lemma \ref{lem:omega:turns}.
According to Figure \ref{fig:turns} and Lemma \ref{lem:direction:well:defined} (2), the only turns with a negative sign are of the form $(\epsilon, \epsilon')$ or $(\epsilon'^{-1}, \epsilon^{-1})$, where $\epsilon'$ is an uppermost edge of some sector. The list of all such turns in $\B_0$ can be found using Figure \ref{fig:eesr_sectors_graphs}. From this we deduce that $\left(\epsilon_{15}^{-1}, \epsilon_{12}^{-1} \right)$ is the only turn of a negative sign in $\zeta_1$, and that $\left(\epsilon_{10}, \epsilon_{11}\right)$ is the only turn of a negative sign in $\zeta_{10}$. Hence $\mathring{\omega}(\zeta_1) = \mathring{\omega}(\zeta_{10}) = -1$.
\end{proof}

\subsection{The taut polynomials of surgered triangulations}
Now we are ready to prove that the taut polynomial of $\V_k$ vanishes.
\begin{proposition}\label{prop:taut:vanishes}
For every integer $k\geq 0$ the taut polynomial $\Theta_k$ of the veering triangulation~$\V_k$ vanishes.
\end{proposition}
\begin{proof}
	The proof was already outlined in Section \ref{sec:strategy}. Here we just fill in the details using the notation introduced in Sections \ref{sec:blowing-up} -- \ref{sec:orientation:homo} and previous lemmas.
	
Let $\langle X \ | \ R \rangle$ be the presentation \eqref{eqn:drilled:presentation} of $\pi_1(\mathring{M})$. A presentation $\langle X \ | \ R_k \rangle$ of $\pi_1(M_k)$ is obtained by adding one relation $s_k = \mu\lambda^k = \epsilon_{21}\epsilon_{17}^{-1}\epsilon_{18}^{-1}\epsilon_{23}^{-1}(\epsilon_{18}\epsilon_{19})^k$ to $R$.
Since the Fox derivative of a word $u$ relative to a generator $\zeta_i$ does not depend on the relations in the group generated by~$X$ (see Section \ref{sec:computation:fox}), the Jacobian $J_k$ of $\langle X \ | \ R_k \rangle$ is obtained from the Jacobian $\mathring{J}$ of $\langle X \ | \ R \rangle$ by adding one row $\frac{\partial s_k}{\partial X}$. 

By Theorem \ref{thm:taut:is:twisted} 
\[\Theta_k = \Delta_{M_k}^{\omega_k \otimes \pi_k}.\]
Since $|X| = 3$, it follows that $\Theta_k$ is the greatest common divisor of the $2\times 2$ minors of the matrix $J_k^{\omega_k\otimes \pi_k}$ obtained from $J_k$ by mapping its entries through \[\omega_k\otimes \pi_k: \zz\lbrack \pi_1(M_k)\rbrack  \rightarrow \zz \lbrack H_k \rbrack.\] By Lemma \ref{lem:omega} and the definition \eqref{eqn:projection:to:Hk:def} of $\mathring{\pi}_k$, for every word $u$ in $X$ we have
\[(\mathring{\omega}\otimes\mathring{\pi}_k)(u) = (\omega_k\otimes \pi_k)(u).\]
Consequently, $J_k^{\omega_k\otimes \pi_k}$ is obtained from $\mathring{J}^{\mathring{\omega}\otimes \mathring{\pi}_k}$ by adding one row. The entries of the Jacobian $\mathring{J}$, and their images under $\mathring{\omega}\otimes \mathring{\pi}_k$, are listed in Appendix \ref{sec:AppendixA}. We have
\[\mathring{J}^{\mathring{\omega}\otimes \mathring{\pi}_k} = \begin{bmatrix}
0 & 0 & 0\\
0 & 0 & 0
\end{bmatrix}.\]
Therefore regardless of $k$ the greatest common divisor of the $2\times 2$ minors of the matrix~$J_k^{\omega_k\otimes \pi_k}$ is equal to 0.
\end{proof}
We have completed the proof of the Main Theorem.
\begin{thm:main}\label{thm:main}
	For any integer $k\geq0$ there is a veering triangulation $\V_k$ with $8+k$ tetrahedra whose taut polynomial vanishes.\qed
\end{thm:main}

\begin{remark}\label{rem:other:sequence}
In Section \ref{sec:initial} we observed that there are two vertical surgery curves in $\B_0$ that can potentially be used to construct infinitely many veering triangulations with a vanishing taut polynomial. Indeed, the other choice results in a different sequence $\left(\V_k'\right)_{k\geq0}$ of veering triangulations with $8+k$ tetrahedra whose taut polynomials vanish. The first triangulation $\V_1'$  has the taut signature \texttt{jLLAMLQacdefgihiijkaqqhuuwj\_200210022} which is different than the taut signature of $\V_1$ from the Main Thereorem; see Table \ref{tab:signatures}. Thus for $k \geq 0$ a veering triangulation with $8+k$ tetrahedra and a vanishing taut polynomial  is not unique.
	\end{remark}
\begin{remark}\label{rem:double:covers}
It follows from Lemmas \ref{lem:omega} and \ref{lem:omega:final} that for every $k \geq 0$ the stable branched surface of $\V_k$ is not orientable. The same is true for the veering triangulations~$\left(\V_k'\right)_{k\geq0}$ from Remark \ref{rem:other:sequence}. However,  nonorientability of the stable branched surface is not a necessary condition on a veering triangulation for its taut polynomial to vanish. Using the same methods as we used in Proposition~\ref{prop:taut:vanishes} one can show that   the lift $\widehat{\V}_k$ of~$\V_k$ to the double cover of $M_k$ determined by $\omega_k$ has a vanishing taut polynomial. This construction yields an infinite sequence of veering triangulations with the number of tetrahedra tending to infinity, orientable stable branched surfaces, and a vanishing taut polynomial. By Theorem \ref{thm:taut:is:twisted}, the Alexander polynomials of 3-manifolds underlying $(\widehat{\V}_k)_{k \geq 0}$ also vanish. 
\end{remark}
\subsection{The flows encoded by $(\V_k)_{k \geq 0}$} \label{sec:the:flows}
A veering triangulation $\V$ of $M$  equips the boundary components of $M$ with a train track which decomposes into finitely many \emph{ladders} separated by \emph{ladderpole curves} \cite[Section 1.2]{Gueritaud_CT}.  Suppose that~$M$ has one boundary component $T$. Let $l$ be the union of all ladderpole curves in $T$, $s$ be a slope on~$T$, and $p$ be  half of the geometric intersection number between $l$ and~$s$. By a result of Agol-Tsang, if $p\geq 2$ then~$\V$ determines a transitive pseudo-Anosov flow $\Psi(s)$ on the 3-manifold $M(s)$ obtained from $M$ by Dehn filling it along~$s$  \cite[Theorem 5.1]{Tsang-Agol}. The core of the filling solid torus is a $p$-pronged orbit of $\Psi(s)$. Therefore the same veering triangulation encodes infinitely many pseudo-Anosov flows: one for almost every Dehn filling slope. Of particular interest are slopes of boundary components of surfaces carried by $\V$ (if they exist). The manifold filled along such a slope admits an embedded surface almost transverse to $\Psi(s)$ \cite[Theorem~5.1]{LMT_flow}.

For every $k \geq0$ the veering triangulation $\V_k$ from the Main Theorem determines two ladders on the unique boundary torus $T_k$ of $M_k$; a picture can be found in the Veering Census \cite{VeeringCensus}. Surfaces carried by $\V_k$ are all twice-punctured tori. 
Let $S_k$ be a twice-punctured torus carried by $\V_k$. The geometric intersection number between a  connected component  of $T_k \cap S_k$ and the union of ladderpole curves in $T_k$ is equal to four. Thus $\V_k$ determines an Anosov flow  on~$M_k(s_k)$.  This flow is transverse to a torus but non-circular, i.e. the torus is not a fiber of a fibration over the circle.
Using Regina \cite{regina} we have found that~$M_k(s_k)$ is a closed 3-manifold obtained from the Seifert fibered space over an annulus with two exceptional fibers, both  with parameters $(2,1)$, by identifying its two boundary components; the identification homeomorphism is different for different~$k$. 


Not all veering triangulations with a vanishing taut polynomial have only two ladders. For instance, veering triangulations from Remark \ref{rem:double:covers} have four ladders. They also determine Anosov flows transverse to tori. Furthermore, there are  veering triangulations with more than four ladders and a vanishing taut polynomial. Possible examples are veering triangulations with taut signatures 
\begin{center} \texttt{oLLLvPLQAQccefjimijmllnnlniitcvikvhlljrcj\_10221211112122}, \\
	 \texttt{pLLLLAzMMMQbceighjkljlmonootsajsgalqlbgalbk\_201022222110102}.
	\end{center}
The  first one lives on a 3-manifold with two boundary components; it determines eight ladders in one boundary component and four in the other. The second one determines six ladders in the unique boundary component of its underlying manifold. All  statements from the last two paragraphs can be verified using tools available in~\cite{VeeringGitHub}. 

\subsection{Further questions}
	The Main Theorem motivates the following questions.

\begin{question}
	Let $\Theta$ be the taut polynomial of a veering triangulation $\V$. 
	\begin{enumerate}
		\item Are there infinitely many veering triangulations with the taut polynomial equal to $\Theta$? 
		\item Can an arbitrarily large veering triangulation have the taut polynomial equal to $\Theta$?
	\end{enumerate}
\end{question}
A positive answer to these questions  for $\Theta  \neq 0$ would provide further evidence that the knowledge of all possible twisted Alexander polynomials $\Delta_M^{\omega \otimes \pi}$, for all homomorphisms $\omega: \pi_1(M)\rightarrow \lbrace - 1, 1 \rbrace$, is not sufficient to bound the number of tetrahedra of a veering triangulation of $M$. 

Furthermore, it would be interesting to know whether pseudo-Anosov flows encoded by veering triangulations with the same taut polynomial share any important dynamical properties.

\newpage

	\bibliographystyle{abbrv}
\bibliography{mybib}
\appendix
\newpage 
\section{The Jacobian of \eqref{eqn:drilled:presentation}}\label{sec:AppendixA}
\[X = \lbrace \zeta_1, \zeta_{10}, \zeta_{21} \rbrace\]
\begin{gather*}
r_1 = \zeta_{21}^{-2}\zeta_1^{-1}\left(\zeta_{10}^{-1}\zeta_{21}^2\zeta_{10}\zeta_1^{-1}\zeta_{10}^{-1}\zeta_{21}^2\right)^2\zeta_{10}\zeta_1^{-1}\zeta_{21}^{-2}\zeta_{10}, \\ r_2 = \zeta_{21}^{-2}\left(\zeta_{10}\zeta_1^{-1}\zeta_{10}^{-1}\zeta_{21}^2\right)^2\zeta_1\zeta_{21}^2\zeta_1^{-1}\left(\zeta_{21}^{-2}\zeta_{10}\zeta_1\zeta_{10}^{-1}\right)^2 
\end{gather*}

The entries $\frac{\partial r_j}{\partial\zeta_j}$ of the Jacobian of $\pi_1(\mathring{M}) = \langle X \ | \ r_1, r_2 \rangle$ are listed below. For clarity, each term is written in a separate line. On the right side of the vertical line we compute the images of terms under $\mathring{\omega} \otimes \mathring{\pi}_k:\zz \lbrack \pi_1(\mathring{M}) \rbrack \rightarrow \zz\lbrack H_k \rbrack$ using Lemmas \ref{lem:projection:to:H_k} and \ref{lem:omega:final}.

\renewcommand{\arraystretch}{1.5}
\[
\begin{array}{ll|l}\frac{\partial r_1}{\zeta_1} = &-\zeta_{21}^{-2} \zeta_{1}^{-1} &+1\\
&- \zeta_{21}^{-2} \zeta_{1}^{-1} \zeta_{10}^{-1} \zeta_{21}^{2} \zeta_{10} \zeta_{1}^{-1} &-1\\
&- \zeta_{21}^{-2} \zeta_{1}^{-1} \zeta_{10}^{-1} \zeta_{21}^{2} \zeta_{10} \zeta_{1}^{-1} (\zeta_{10}^{-1}\zeta_{21}^{2})^{2} \zeta_{10} \zeta_{1}^{-1} &-c^{-1} \\ 
&-\zeta_{21}^{-2} \zeta_{1}^{-1} (\zeta_{10}^{-1} \zeta_{21}^{2} \zeta_{10} \zeta_{1}^{-1} \zeta_{10}^{-1} \zeta_{21}^{2})^{2} \zeta_{10} \zeta_{1}^{-1}&+c^{-1}\\
& & \\[-0.5em]
\frac{\partial r_1}{\zeta_{10}} =& -\zeta_{21}^{-2} \zeta_1^{-1} \zeta_{10}^{-1} &-c^{-1}\\
&+ \zeta_{21}^{-2} \zeta_1^{-1} \zeta_{10}^{-1} \zeta_{21}^{2} &-c^{-1}\\
&- \zeta_{21}^{-2} \zeta_1^{-1} \zeta_{10}^{-1} \zeta_{21}^{2} \zeta_{10} \zeta_1^{-1} \zeta_{10}^{-1} &+c^{-1} \\
&-\zeta_{21}^{-2} \zeta_1^{-1} \zeta_{10}^{-1} \zeta_{21}^{2} \zeta_{10} \zeta_1^{-1} \zeta_{10}^{-1} \zeta_{21}^{2} \zeta_{10}^{-1} &-c^{-2}\\
&+ \zeta_{21}^{-2} \zeta_1^{-1} \zeta_{10}^{-1} \zeta_{21}^{2} \zeta_{10} \zeta_1^{-1} (\zeta_{10}^{-1} \zeta_{21}^{2})^{2} &-c^{-2}\\
&-\zeta_{21}^{-2} \zeta_1^{-1} \zeta_{10}^{-1} \zeta_{21}^{2} \zeta_{10} \zeta_1^{-1} (\zeta_{10}^{-1} \zeta_{21}^{2})^{2} \zeta_{10} \zeta_1^{-1} \zeta_{10}^{-1} &+c^{-2}\\
&+ 
\zeta_{21}^{-2} \zeta_1^{-1} (\zeta_{10}^{-1} \zeta_{21}^{2} \zeta_{10} \zeta_1^{-1} \zeta_{10}^{-1} \zeta_{21}^{2})^{2} &+c^{-2} \\
&+\zeta_{21}^{-2} \zeta_1^{-1} (\zeta_{10}^{-1} \zeta_{21}^{2} \zeta_{10} \zeta_1^{-1} \zeta_{10}^{-1} \zeta_{21}^{2})^{2} \zeta_{10} \zeta_1^{-1} \zeta_{21}^{-2}&+c^{-1}\\
& & \\[-0.5em]
\frac{\partial r_1}{\zeta_{21}} = &-\zeta_{21}^{-1} &+1\\
&- \zeta_{21}^{-2} &-1\\
&+ \zeta_{21}^{-2} \zeta_1^{-1} \zeta_{10}^{-1}&-c^{-1} \\
&+ \zeta_{21}^{-2} \zeta_1^{-1} \zeta_{10}^{-1} \zeta_{21} &+c^{-1}\\
&+ \zeta_{21}^{-2} \zeta_1^{-1} \zeta_{10}^{-1} \zeta_{21}^{2} \zeta_{10} \zeta_1^{-1} \zeta_{10}^{-1} &+c^{-1}\\
&+ \zeta_{21}^{-2} \zeta_1^{-1} \zeta_{10}^{-1} \zeta_{21}^{2} \zeta_{10} \zeta_1^{-1} \zeta_{10}^{-1} \zeta_{21} &-c^{-1}\\
&+\zeta_{21}^{-2} \zeta_1^{-1} \zeta_{10}^{-1} \zeta_{21}^{2} \zeta_{10} \zeta_1^{-1} \zeta_{10}^{-1} \zeta_{21}^{2} \zeta_{10}^{-1} &-c^{-2}\\
&+\zeta_{21}^{-2} \zeta_1^{-1} \zeta_{10}^{-1} \zeta_{21}^{2} \zeta_{10} \zeta_1^{-1} \zeta_{10}^{-1} \zeta_{21}^{2} \zeta_{10}^{-1} \zeta_{21} &+c^{-2}\\
&+\zeta_{21}^{-2} \zeta_1^{-1} \zeta_{10}^{-1} \zeta_{21}^{2} \zeta_{10} \zeta_1^{-1} (\zeta_{10}^{-1} \zeta_{21}^{2})^{2} \zeta_{10} \zeta_1^{-1} \zeta_{10}^{-1}  &+c^{-2}\\
&+\zeta_{21}^{-2} \zeta_1^{-1} \zeta_{10}^{-1} \zeta_{21}^{2} \zeta_{10} \zeta_1^{-1} (\zeta_{10}^{-1} \zeta_{21}^{2})^{2} \zeta_{10} \zeta_1^{-1} \zeta_{10}^{-1} \zeta_{21}  &-c^{-2}\\
&-\zeta_{21}^{-2} \zeta_1^{-1} (\zeta_{10}^{-1} \zeta_{21}^{2} \zeta_{10} \zeta_1^{-1} \zeta_{10}^{-1} \zeta_{21}^{2})^{2} \zeta_{10} \zeta_1^{-1} \zeta_{21}^{-1} &-c^{-1}\\
&-\zeta_{21}^{-2} \zeta_1^{-1} (\zeta_{10}^{-1} \zeta_{21}^{2} \zeta_{10} \zeta_1^{-1} \zeta_{10}^{-1} \zeta_{21}^{2})^{2} \zeta_{10} \zeta_1^{-1} \zeta_{21}^{-2}&+c^{-1}
\end{array}\]

\[\begin{array}{ll|l}
\frac{\partial r_2}{\zeta_1} = &-\zeta_{21}^{-2} \zeta_{10} \zeta_{1}^{-1} &-c\\
&- \zeta_{21}^{-2} \zeta_{10} \zeta_{1}^{-1} \zeta_{10}^{-1} \zeta_{21}^{2} \zeta_{10} \zeta_{1}^{-1} &+c\\
&+ \zeta_{21}^{-2} (\zeta_{10} \zeta_{1}^{-1} \zeta_{10}^{-1} \zeta_{21}^{2})^{2} &+1\\
&- \zeta_{21}^{-2} (\zeta_{10} \zeta_{1}^{-1} \zeta_{10}^{-1} \zeta_{21}^{2})^{2} \zeta_{1} \zeta_{21}^{2} \zeta_{1}^{-1}&-1 \\
&+ \zeta_{21}^{-2} (\zeta_{10} \zeta_{1}^{-1} \zeta_{10}^{-1} \zeta_{21}^{2})^{2} \zeta_{1} \zeta_{21}^{2} \zeta_{1}^{-1} \zeta_{21}^{-2} \zeta_{10} &+c\\
&+ \zeta_{21}^{-2} (\zeta_{10} \zeta_{1}^{-1} \zeta_{10}^{-1} \zeta_{21}^{2})^{2} \zeta_{1} \zeta_{21}^{2} \zeta_{1}^{-1} \zeta_{21}^{-2} \zeta_{10} \zeta_{1} \zeta_{10}^{-1} \zeta_{21}^{-2} \zeta_{10}&-c\\
& & \\[-0.5em]
\frac{\partial r_2}{\zeta_{10}} = &+\zeta_{21}^{-2} &+1\\
&- \zeta_{21}^{-2} \zeta_{10} \zeta_{1}^{-1} \zeta_{10}^{-1} &+1\\
&+ \zeta_{21}^{-2} \zeta_{10} \zeta_{1}^{-1} \zeta_{10}^{-1} \zeta_{21}^{2} &-1\\
&- \zeta_{21}^{-2} \zeta_{10} \zeta_{1}^{-1} \zeta_{10}^{-1} \zeta_{21}^{2} \zeta_{10} \zeta_{1}^{-1} \zeta_{10}^{-1} &-1\\
&+ \zeta_{21}^{-2} (\zeta_{10} \zeta_{1}^{-1} \zeta_{10}^{-1} \zeta_{21}^{2})^{2} \zeta_{1} \zeta_{21}^{2} \zeta_{1}^{-1} \zeta_{21}^{-2} &+1\\
&- \zeta_{21}^{-2} (\zeta_{10} \zeta_{1}^{-1} \zeta_{10}^{-1} \zeta_{21}^{2})^{2} \zeta_{1} \zeta_{21}^{2} \zeta_{1}^{-1} \zeta_{21}^{-2} \zeta_{10} \zeta_{1} \zeta_{10}^{-1}&+1 \\
&+ \zeta_{21}^{-2} (\zeta_{10} \zeta_{1}^{-1} \zeta_{10}^{-1} \zeta_{21}^{2})^{2} \zeta_{1} \zeta_{21}^{2} \zeta_{1}^{-1} \zeta_{21}^{-2} \zeta_{10} \zeta_{1} \zeta_{10}^{-1} \zeta_{21}^{-2} &-1\\
&- \zeta_{21}^{-2} (\zeta_{10} \zeta_{1}^{-1} \zeta_{10}^{-1} \zeta_{21}^{2})^{2} \zeta_{1} \zeta_{21}^{2} \zeta_{1}^{-1} (\zeta_{21}^{-2} \zeta_{10} \zeta_{1} \zeta_{10}^{-1})^{2}&-1\\
& & \\[-0.5em]
\frac{\partial r_2}{\zeta_{21}} =& -\zeta_{21}^{-1} &+1\\
&- \zeta_{21}^{-2} &-1\\
&+ \zeta_{21}^{-2} \zeta_{10} \zeta_{1}^{-1} \zeta_{10}^{-1} &-1\\
&+ \zeta_{21}^{-2} \zeta_{10} \zeta_{1}^{-1} \zeta_{10}^{-1} \zeta_{21} &+1\\
&+ \zeta_{21}^{-2} \zeta_{10} \zeta_{1}^{-1} \zeta_{10}^{-1} \zeta_{21}^{2} \zeta_{10} \zeta_{1}^{-1} \zeta_{10}^{-1}&+1 \\
&+ \zeta_{21}^{-2} \zeta_{10} \zeta_{1}^{-1} \zeta_{10}^{-1} \zeta_{21}^{2} \zeta_{10} \zeta_{1}^{-1} \zeta_{10}^{-1} \zeta_{21} &-1\\
&+ \zeta_{21}^{-2} (\zeta_{10} \zeta_{1}^{-1} \zeta_{10}^{-1} \zeta_{21}^{2})^{2} \zeta_{1} &-1\\
&+ \zeta_{21}^{-2} (\zeta_{10} \zeta_{1}^{-1} \zeta_{10}^{-1} \zeta_{21}^{2})^{2} \zeta_{1} \zeta_{21} &+1\\
&- \zeta_{21}^{-2} (\zeta_{10} \zeta_{1}^{-1} \zeta_{10}^{-1} \zeta_{21}^{2})^{2} \zeta_{1} \zeta_{21}^{2} \zeta_{1}^{-1} \zeta_{21}^{-1} &+1\\
&- \zeta_{21}^{-2} (\zeta_{10} \zeta_{1}^{-1} \zeta_{10}^{-1} \zeta_{21}^{2})^{2} \zeta_{1} \zeta_{21}^{2} \zeta_{1}^{-1} \zeta_{21}^{-2} &-1\\
&- \zeta_{21}^{-2} (\zeta_{10} \zeta_{1}^{-1} \zeta_{10}^{-1} \zeta_{21}^{2})^{2} \zeta_{1} \zeta_{21}^{2} \zeta_{1}^{-1} \zeta_{21}^{-2} \zeta_{10} \zeta_{1} \zeta_{10}^{-1} \zeta_{21}^{-1} &-1\\
&- \zeta_{21}^{-2} (\zeta_{10} \zeta_{1}^{-1} \zeta_{10}^{-1} \zeta_{21}^{2})^{2} \zeta_{1} \zeta_{21}^{2} \zeta_{1}^{-1} \zeta_{21}^{-2} \zeta_{10} \zeta_{1} \zeta_{10}^{-1} \zeta_{21}^{-2}&+1
\end{array}\]

\vfill

\end{document}